\documentclass[3p,times]{elsarticle}

\usepackage{amssymb}

\usepackage{amsthm}

\usepackage{amsmath}

\usepackage{nameref}

\usepackage{kotex}

\usepackage{enumitem}

\usepackage[dvipsnames]{xcolor}

	\newcommand{\DpD}[1]{ |#1 |^{p-2} ( #1 ) }
	
	\newcommand{\DpL}[0]{\Delta_{p,\psi}}

	\newcommand{\Ical}[0]{\mathcal{I}}
	\newcommand{\Jcal}[0]{\mathcal{J}}
	\newcommand{\Kcal}[0]{\mathcal{K}}

	\newcommand{\sI}[1]{\sum_{#1\in\Ical}}

	\newcommand{\csw}[0]{\psi(\Lnorm{x_j - x_i} )}

	\newcommand{\Rcwt}[0]{\psi_R (\Lnorm{x_j (t) - x_i (t)} )}
	\newcommand{\Rcwtzero}[0]{\psi_R (\Lnorm{x_j (0) - x_i (0)} )}

	\newcommand{\norm}[2]{\left\| #1 \right\|_{#2}}
	
\newtheorem{theorem}[subsection]{Theorem}
\newtheorem{lemma}[subsection]{Lemma}

\newtheorem{proposition}[subsection]{Proposition}

\theoremstyle{definition}

\newtheorem{remark}[subsection]{Remark}
\newtheorem{example}[subsection]{Example}
\newtheorem{TODO}[]{TODO}

\begin{document}

\begin{frontmatter}

\title{Finite flocking time of the nonlinear Cucker--Smale model with Rayleigh friction type using the discrete $p$-Laplacian}

\journal{...}

\author[JHK]{Jong-Ho Kim}
\ead{jhkim@nims.re.kr}
\address[JHK]{National Institute for Mathematical Sciences, Daejeon-si, Republic of Korea}

\author[YJL]{Young Ju Lee}
\ead{yjlee@txstate.edu}
\address[YJL]{Department of Mathematics, Texas State University, San Marcos, TX, USA}

\author[JHP]{Jea-Hyun Park\corref{JHPark}\fnref{JHPT}}
\ead{parkjhm@kunsan.ac.kr}
\cortext[JHPark]{Corresponding Author}
\address[JHP]{Department of Mathematics, Kunsan National University, Gunsan-si, Republic of Korea}

\begin{abstract}

The study of collective behavior in multi-agent systems has attracted the attention of many researchers due to its wide range of applications. Among them, the Cucker-Smale model was developed to study the phenomenon of flocking, and various types of extended models have been actively proposed and studied in recent decades.

In this study, we address open questions of the Cucker--Smale model with norm-type Rayleigh friction:
{\bf (i)} The positivity of the communication weight,
{\bf (ii)} The convergence of the norm of the velocities of agents,
{\bf (iii)} The direction of the velocities of agents.
For problems (i) and (ii), we present the nonlinear Cucker--Smale model with norm-type Rayleigh friction, where the nonlinear Cucker--Smale model is generalized to a nonlinear model by applying a discrete $p$-Laplacian operator.
For this model, we present conditions that guarantee that the norm for velocities of agents converges to 0 or a positive value, and we also show that the regular communication weight satisfies the conditions given in this study. In particular, we present a condition for the initial configuration to obtain that the norm of agent velocities converges to only some positive value. 

By contrast, problem (iii) is not solved by the norm-type nonlinear model.
Thus, we propose a nonlinear Cucker–Smale model with a vector-type Rayleigh friction for problem (iii).
In parallel to the first model, we show that the direction of the agents’ velocities can be controlled by parameters in the nonlinear Cucker–Smale model with the vector-type Rayleigh friction.

\end{abstract}


\end{frontmatter}

\newcommand{\MDdpl}[0]{\Delta_{p,\psi}}

\newcommand{\Mdpdiff}[1]{ |#1 |^{p-2} ( #1 ) }

\newcommand{\sumI}[1]{\sum_{#1\in\Ical}}
\newcommand{\sumJ}[1]{\sum_{#1\in\Jcal}}
\newcommand{\sumK}[1]{\sum_{#1\in\Kcal}}

\newcommand{\CSweight}[0]{\psi(\Lnorm{x_j – x_i} )}
\newcommand{\CSweightT}[0]{\psi(\Lnorm{x_j(t) – x_i(t)} )}

\newcommand{\Inpro}[2]{\Big\langle #1, #2\Big\rangle}

\newcommand{\LN}[2]{\| #1 \|_{#2}}
\newcommand{\Lnorm}[1]{\| #1 \|_2}
\newcommand{\Lpnorm}[1]{\| #1 \|_p}
\newcommand{\Lqnorm}[1]{\| #1 \|_q}
\newcommand{\Lrnorm}[1]{\| #1 \|_r}

\newcommand{\UC}[0]{C_M}
\newcommand{\DC}[0]{C_m}

\newcommand{\MSys}[0]{\eqref{Main_eq(1)}-\eqref{Main_eq(2)}}
\newcommand{\MaineqR}[1]{\MDdpl #1 + a_k \varphi_q(#1 ) – b_k \varphi_r(#1 ) }

\newcommand{\pmin}[0]{\psi_{\min}}
\newcommand{\pmax}[0]{\psi_{\max}}


\label{====================================}
\section{Introduction}\label{Section: int}

\newcommand{\Cab}[0]{\left( \frac{a}{b} \right)^{\frac{1}{r-q}}}
\newcommand{\Cabk}[0]{\left( \frac{a_k}{b_k} \right)^{\frac{1}{r-q}}}

In recent years, there has been a growing interest among researchers in the development and analysis of mathematical models for multi-agent systems (MAS) in fields including biology, social sciences, physics, and engineering. The modeling of MAS as a system of mathematical differential equations has emerged as an active research area, providing a powerful tool for comprehending complex phenomena such as the flocking of birds, schooling of fish, and colonization of bacteria. This improved understanding can also facilitate the development of more efficient engineering systems, including unmanned aerial vehicles. Consequently, various models, such as the Vicsek model, Kuramoto model, minimal model, and Cucker--Smale model, have been proposed and are being studied in different fields
(see \cite{argun2021vicsek,bae2022global,kim2022clustering,kim2023analysis,li2022modeling,rodrigues2016kuramoto} and the references therein).

In this study, we are interested in the Cucker--Smale model with Rayleigh friction.
The Cucker--Smale model (C-S model) was first introduced in \cite{cucker2007emergent} to study the collective and self-driven motion of self-propelled agents.
The mathematical model is given by 
\begin{align*}
  	&\frac{d x_{i,k}}{dt}(t) = v_{i,k}(t),\\
	&\frac{d v_{i,k}}{dt}(t) = \sum_{j=1}^N \psi_R(\| x_j(t) - x_i(t) \|_2^2) (v_{j,k}(t) - v_{i,k}(t)),
	\quad
	i\in \Ical:=\{ 1,\ldots, N\}, ~k\in \Kcal:=\{ 1, \ldots,d\},
\end{align*}
where $N$ is the number of agents, $a$ is a positive coupling strength,
$x_i = (x_{i,1}, \ldots x_{i,d}) \in \mathbb{R}^d$, and $v_i = (v_{i,1}, \ldots v_{i,d}) \in \mathbb{R}^d$ are the position and velocity of the $i$th agent in phase space $\mathbb{R}^d \times \mathbb{R}^d$, respectively,
$\psi_R$ is a regular communication weight defined by $\psi_R(s) = \frac{K}{\left( 1+s \right)^{\beta}}$ for some $\beta>0$ and $K>0$, 
and $\| \cdot \|_2$ is the Euclidean norm.
For this model, Cucker and Smale presented conditions to ensure that agents' velocities converge to a common velocity and that the maximum distance between agents is uniformly bounded, i.e., the so-called ``{\it asymptotic flocking}’’ is defined by 
\begin{itemize}
	\item[(i)] The relative velocity fluctuations tend to zero as time tends to infinity (velocity alignment, or consensus):
	\begin{align*}
  		\lim_{t\to\infty} \norm{v_j (t) - v_i (t)}{2} =0, \text{ for all }i , j \in \Ical,
	\end{align*}
	\item[(ii)] The diameter of a group is uniformly bounded in time t (forming a group): 
	\begin{align*}
  		\sup_{t\geq 0} \norm{x_j (t) - x_i (t)}{2} < \infty\text{ for all }i , j \in \Ical.
	\end{align*}
\end{itemize}
This C-S model was later extended into various forms depending on the purpose of each study, such as, 
(mono-cluster) flocking \cite{cucker2007emergent,ha2017critical,ha2009simple,motsch2011new}, 
multi-cluster flocking 
\cite{cho2016emergence,cho2016emergence_unit_speed},
pattern formation
\cite{choi2019collisionless,perea2009extension}, 
collision-avoidance
\cite{ahn2012collision,bongini2014sparse,carrillo2017sharp,cucker2014conditional}, 
leadership
\cite{dalmao2011cucker,dong2016flocking,shen2008cucker}, and 
forcing terms and control
\cite{bailo2018optimal,ha2010asymptotic}.

In particular, among these models, the C-S model with Rayleigh friction \cite{ha2010asymptotic} caught our attention. 
In \cite{ha2010asymptotic}, Ha and his colleagues introduced the C-S model with Rayleigh friction:
\begin{align}
  	&\frac{d x_{i,k}}{dt}(t) = v_{i,k}(t), \label{Ha eq (1)} \\
	&\frac{d v_{i,k}}{dt}(t) = \frac{\lambda}{N} \sum_{j=1}^N \psi(\| x_j(t) - x_i(t) \|_2^2) (v_{j,k}(t) - v_{i,k}(t)) + \delta  v_{i,k} ( 1- \norm{v_i(t)}{2}^2 ),\label{Ha eq (2)}
\end{align}
for $i\in\Ical$, $k\in\Kcal$, $\lambda>0$, $\delta>0$, and $\psi:\mathbb{R} \to \mathbb{R}$ is a communication weight function satisfying
\begin{align}\label{Ha condition (1)}
	\psi(s) \geq \psi_* >0, \quad \text{for some positive constant } \psi_*
\end{align}
to investigate how Rayleigh-type friction affects the dynamics of the C-S system.
They proved that if $\lambda \psi_* > \delta$, the system \eqref{Ha eq (1)}-\eqref{Ha eq (2)} induces the flocking behavior of agents. Moreover, the norm of velocities converges to 0 or 1, that is, $\norm{v_i(t)}{2} \to 0 \text{ or } 1$ as $t\to\infty$ for all $i\in\Ical$.

In particular, we note that some open problems have been discussed in \cite{ha2010asymptotic}:
\begin{itemize}
	\item[(P1)] (Problem for the condition $\psi(s) \geq \psi_* >0$, $s\in\mathbb{R}$) 
	Since they assumed $\psi(s) \geq \psi_* >0$, $s\in\mathbb{R}$, their results cannot be applied for the algebraically decaying communication C-S weight with $\psi_*=0$, such as the regular communication weight $\psi_R$.
	Therefore, we must determine a condition weaker than condition \eqref{Ha condition (1)}.
	\item[(P2)] (Problem for the convergence of $\norm{v_i}{2}$)
	For this problem, they assumed that it is not apriori clear which initial configuration converges to 0 or 1.
	\item[(P3)] (Problem for the direction of velocity)
	From the results in \cite{ha2010asymptotic}, we only know the convergence of agents' speed (i.e., $\norm{v_i}{2}$), not their velocities.
\end{itemize}

The aim of this paper was to address these open problems using a nonlinear operator called the discrete $p$-Laplacian defined by
\begin{align*}
	\DpL v_{i,k} = \sum_{j=1}^N \psi\left( \| x_j - x_i\|_2 \right)
					\DpD{v_{j,k} - v_{i,k}},
\end{align*}
where $p>1$, and $\psi : [0,\infty) \to [0,\infty)$ is a non-increasing and differentiable function.
We applied the discrete $p$-Laplacian because of its properties. To the best of our knowledge, the discrete Laplacian (that is, $p=2$) only provides information about the asymptotic behavior of the agents owing to technical reasons. For example, the norm of $v_i$ decays exponentially; however, we do not know whether $\norm{v_i}{2}=0$ in finite time. Furthermore, if $1<p<2$, due to the discrete $p$-Laplacian, agents’ behavioral properties, such as flocking and consensus, can be revealed in finite time (see \cite{kim2020complete}).
We intend to use these features to address the open problems mentioned above.
Therefore, we consider a nonlinear C-S model with (norm-type) Rayleigh friction applying the discrete $p$-Laplacian:
\begin{align}
	&\frac{dx_{i,k}}{dt} = v_{i,k}    \label{Intro Rayleigh friction (1)}\\
	&\frac{dv_{i,k}}{dt} = \DpL v_{i,k} + a v_{i,k} \norm{v_i}{2}^{q-2}
	- b v_{i,k} \norm{v_i}{2}^{r-2}, \label{Intro Rayleigh friction (2)}
\end{align}
where $p>1$, $q>1$, $r>1$, $a>0$, $b>0$, and $x_i = (x_{i,1},\ldots,x_{i,d})$ and $v_i = (v_{i,1},\ldots,v_{i,d})$ are the position and velocity of $i$-th particle in $\mathbb{R}^d$, respectively.
Here, for $p=2$, $q=2$, $r=4$, and $a=b=\delta$, the system \eqref{Intro Rayleigh friction (1)}-\eqref{Intro Rayleigh friction (2)} implies the C-S model with Rayleigh friction \eqref{Ha eq (1)}-\eqref{Ha eq (2)}.

For the model \eqref{Intro Rayleigh friction (1)}-\eqref{Intro Rayleigh friction (2)} in this study, we first show that if $p>1$, $2\leq q<r$, 
and the initial configuration for the agents' velocities is non-negative and non-zero, then the system \eqref{Intro Rayleigh friction (1)}-\eqref{Intro Rayleigh friction (2)} forms an asymptotic flocking and the velocities satisfy $\lim_{t\to\infty} \norm{v_i(t)}{2} = \Cab$.
We present a condition to guarantee that the system \eqref{Intro Rayleigh friction (1)}-\eqref{Intro Rayleigh friction (2)} has a finite flocking time and the velocities satisfy $\lim_{t\to\infty} \norm{v_i(t)}{2} = \Cab$ or $0$ for $1<p<2\leq q <r$, $a>0$, and an arbitrary initial configuration.
The given condition is as follows:
	\begin{align}\label{Intro: Main Condition 010}
  		4 C_m \left( 1 - \frac{p}{2} \right) \int_0^\infty \min_{i,j\in\Ical} \CSweightT dt 
  		> \left( \sum_{k\in\Kcal} \left(\max_{i\in\Ical} v_{i,k}(0) - \min_{i\in\Ical} v_{m,k}(0)\right)^2\right)^{\frac{2-p}{2}}
	\end{align}
where $C_m$ is positive and defined in Section 2, and the finite flocking time indicates that there exists a positive time $T>0$ such that 
$\norm{v_j (t) - v_i (t)}{2} =0$, for all $i , j \in \Ical$, and $t\geq T$.

In particular, we note that the condition \eqref{Intro: Main Condition 010} is weaker than \eqref{Ha condition (1)} presented in \cite{ha2010asymptotic} and the regular communication weight $\psi_R$ satisfies \eqref{Intro: Main Condition 010}, which is proved in Section 5. This expression solves problem one (P1).
Moreover, the existence of the finite flocking time provides a condition for the initial (velocity) configuration to solve problem two (P2); that is, from the presented condition, $\norm{v_i(t)}{2}$ converges to only $\Cab$ as $t\to\infty$.
However, we, unfortunately, do not obtain any information about the direction of the agents' progress (i.e., velocities of agents) for the nonlinear C-S model with (norm-type) Rayleigh friction.
The reason for this is that \eqref{Intro Rayleigh friction (2)} contains the norm term $\norm{\cdot}{2}$; hence, we can only obtain information about $\norm{v_i}{2}$ and not each velocity $v_{i,k}$, $k\in\Kcal$ due to the limitations of technical methods of solving ordinary differential equations.

Therefore, for the open problem (P3), we propose the nonlinear C-S model with (vector-type) Rayleigh friction:
\begin{align}
	&\frac{dx_{i,k}}{dt} = v_{i,k}    \label{Intro Main_eq(1)}\\
	&\frac{dv_{i,k}}{dt} = \DpL v_{i,k} + a_k \varphi_q(v_{i,k}) 
	- b_k \varphi_r(v_{i,k}) \label{Intro Main_eq(2)}
\end{align}
where $a_k>0$, $b_k>0$ for $k\in\Kcal$, and the function $\varphi_\gamma$ is defined by
$\varphi_\gamma(s) := |s|^{\gamma-2}s$, $s\in\mathbb{R}$, $\gamma>1$, to control the direction of the agents' progress.
For this model, we show that for $1<p<2$, $p<q<r$, and $k\in\Kcal$, if the initial velocity configuration is non-negative and non-zero, $\lim_{t\to\infty} v_{i,k}(t) = \Cabk$ for all $i\in\Ical$ and the system \eqref{Intro Main_eq(1)}-\eqref{Intro Main_eq(2)} forms asymptotic flocking. 
We also present conditions for the initial configuration to guarantee the existence of finite flocking time, using which, we can see that $v_{i,k}$ converges to either $\Cabk$, $0$, or $-\Cabk$. 
Finally, we present a condition for the initial velocity configuration to guarantee that $v_{i,k}$ converges to only $\Cabk$.
We note that as shown in the system \eqref{Intro Main_eq(1)}-\eqref{Intro Main_eq(2)} above, the difference from systems \eqref{Intro Rayleigh friction (1)}-\eqref{Intro Rayleigh friction (2)} is that the vector type term $\varphi_\gamma$ is given instead of the norm type term. This is a very simple change. However, it will allow us to obtain microscopic information about the velocity $v_{i,k}$ of each agent using the maxima and minima of $v_{i,k}$, rather than macroscopic information such as the norm.

The remainder of this paper is organized as follows. 
Section 2 introduces the preliminary concepts and useful well-known results. 
Section 3 is devoted to a discussion of the nonlinear C-S model with a norm-type Rayleigh friction, and in Section 4, we address the nonlinear C-S model with the vector-type Rayleigh friction. 
In Section 5, we show that the regular communication weight satisfies the conditions presented in Sections 3 and 4.
Finally, we provide several numerical simulations to confirm our analytical results.

\label{====================================}
\section{Preliminaries}\label{Section: pre}

We start this section with maximum and minimum functions defined by $v_{M,k}(t):=\max_{i\in\Ical} v_{ik}(t)$ and $v_{m,k}(t):=\min_{i\in\Ical} v_{ik}(t)$ for $k\in\Kcal$ and $t\geq 0$.
Let us define two index sets: for each $t\geq0$ and $k\in\Kcal$,
\begin{align*}
	\mathcal I_{M,k}^t : = \{ i\in\Ical ~|~ v_{M,k} (t) = v_{i,k}(t) \}
	\text{ and }
	\mathcal I_{m,k}^t : = \{ i\in\Ical ~|~ v_{m,k} (t) = v_{i,k}(t) \}.
\end{align*}
Since $\DpL v_{M,k}(t) \leq 0$ and $\DpL v_{m,k}(t) \geq 0$ for all $k\in\Kcal$ and $t\geq0$, for $i\in \mathcal I_{M,k}^t$ and $j\in \mathcal I_{m,k}^t$, 
\begin{align*}
	v_{i,k}^\prime
	= &
	\DpL v_{M,k} + a v_{i,k} f(v_i) - b v_{i,k} g(v_i)\\
	\leq &
	   b v_{i,k} f(v_i) \left( \frac{a}{b}  -  \frac{g(v_i)}{f(v_i)} \right),
\end{align*}
and
\begin{align*}
	v_{j,k}^\prime 
	= &
	\DpL v_{m,k} + a v_{j,k} f(v_j) - b v_{j,k} g(v_j)\\
	\geq &
	b v_{j,k} f(v_j) \left( \frac{a}{b}  -  \frac{g(v_j)}{f(v_j)} \right),
\end{align*}
where $f$ and $g$ are specified by norm-type as in \eqref{Intro Rayleigh friction (2)} or vector-type as in \eqref{Intro Main_eq(2)} in accordance with the theme of each section.
Since we always assume that $q<r$, $\frac{g(v)}{f(v)}$ is well-defined for $v \in \mathbb{R}$ and strictly increases and decreases on $(0,\infty)$ and $(-\infty, 0)$, respectively.
Thus, if $v_{i,k}(t)$ is large enough on some time interval $I$ such that $\frac{g(v_i)}{f(v_i)}>\frac{a}{b}$, then $v_{i,k}$ strictly decreases on the interval $I$.
By contrast, if $v_{j,k}$ is small enough on some interval $I$ such that $\frac{g(v_j)}{f(v_j)}>\frac{a}{b}$, then $v_{j,k}$ strictly increases on the interval $I$.
Therefore, $v_{i,k}$ is uniformly bounded with respect to $t\geq0$ for all $i\in\Ical$, $k\in\Kcal$.
Using \eqref{Intro Rayleigh friction (2)} and \eqref{Intro Main_eq(2)}, $v_{i,k}^\prime$ is also uniformly bounded. 
Hence, $v_{i,k}$ is Lipschitz continuous, and $v_{M,k}$ and $v_{m,k}$ are also Lipschitz continuous. Therefore, $v_{M,k}$ and $v_{m,k}$ are differentiable almost everywhere.

Throughout this paper, we assume that for each $k\in\Kcal$,
there exist distinct time series $\{t_l\}_{l=0}^\infty$ and $\{s_l\}_{l=0}^\infty$ such that the indices $M$ and $m$ do not change on $(s_l, s_{l+1})$ and $(t_l, t_{l+1})$, respectively.
Then, it is clear that $v_{M,k} \in C^1((s_{l-1}, s_l))$ and $v_{m,k}\in C^1((t_{l-1}, t_l))$ where $C^1 ((a,b))$ is a class consisting of all differentiable functions whose derivative is continuous on $(a,b)$.
Moreover, since $v_{i,k}$ is continuous with respect to $t$,
\begin{align*}
	\int_a^b v_{M,k}^\prime(s) ds = v_{M,k} (b) - v_{M,k} (a),
	\quad\text{and}\quad
	\int_a^b v_{m,k}^\prime(s) ds = v_{m,k} (b) - v_{m,k} (a).
\end{align*}
We note that $v_{M,k}$ and $v_{m,k}$ may or may not be differentiable at $t=s_l$ and $t=t_l$. 
However, since there always exist $i\in \mathcal I_{M,k}^{s_l}$ and $j \in \mathcal I_{m,k}^{t_l}$ such that $v_{M,k}(s_l) = v_{i,k}(s_l)$ and $v_{m,k}(t_l) = v_{j,k}(t_l)$, we consider $v_{M,k}^\prime(s_l)<\alpha$ and $v_{m,k}^\prime(t_l)<\alpha$  for some $\alpha\in\mathbb{R}$ as $v_{i,k}^\prime(s_l) <\alpha$ and $v_{j,k}^\prime(t_l) <\alpha$ for all $i\in\mathcal I_{M,k}^t$ and $j\in\mathcal I_{m,k}^t$.
The same considerations apply to other inequalities ($>$, $\geq$ and $\leq$).
In addition, if necessary, we consider the derivatives of $v_{M,k}$ and $v_{m,k}$ at $t=s_l$ and $t=t_l$, respectively, as 
\begin{align*}
  	v_{M,k}^\prime(s_l) = \lim_{t\to s_l^+} v_{i,k}^\prime(t), \quad
  	v_{m,k}^\prime(t_l) = \lim_{t\to t_l^+} v_{j,k}^\prime(t)
\end{align*}
for some $i\in \{ i\in\Ical~|~ v_{M,k} (t) = v_{i,k}(t), ~t\in(s_l , s_{l+1} )\}$ and $j\in \{ i\in\Ical~|~ v_{m,k} (t) = v_{i,k}(t), ~t\in(t_l , t_{l+1} )\}$.

We now recall some well-known properties for the $p$-norm $\norm{\cdot}{p}$ without proofs, defined by 
$\norm{\mathbf{a}}{p}:=\sum_{i \in \Ical} a_i$ if $\mathbf a=(a_1, \ldots, a_N)$ is a vector in $\mathbb{R}^N$, 
and $\norm{\mathbf a }{\gamma}:= \left( \sumI{i} \sumK{k} a_{i,k}^\gamma \right)^{\frac{1}{\gamma}}$ if $\mathbf a = (a_{i,k})_{N \times d}$ is a $N\times d$-matrix.
These properties are useful when dealing with norms in this study.

\begin{lemma}[\cite{barrett1994finite}]\label{L-inequalities}
	For $\gamma>1$, $d\geq 1$, and $\delta\geq 0$, there exist $C_1(p,d)>0$ and $C_2(p,d)>0$ such that for all $x, y \in \mathbb{R}^d$,
	\begin{align}\label{L-ineq-eq010}
		\norm{ \norm{x}{2}^{\gamma-2} x -  \norm{y}{2}^{\gamma-2} y   }{2}
		\leq 
		C_1 \norm{x-y}{2}^{1-\delta} 
		\left( \norm{x}{2}+\norm{y}{2} \right)^{\gamma-2+\delta},
	\end{align}
	and
	\begin{align}\label{L-ineq-eq020}
		\Inpro{x-y}{\norm{x}{2}^{\gamma-2} x -  \norm{y}{2}^{\gamma-2} y  }
		\geq
		C_2 \norm{x - y}{2}^{2+\delta} 
		\left( \norm{x}{2}+\norm{y}{2} \right)^{\gamma-2-\delta},
	\end{align}
	where $\Inpro{\cdot}{\cdot}$ is the dot product, and $\norm{\cdot}{2}$ is the Euclidean norm.
\end{lemma}
\begin{proof}
	For a proof of this lemma, see \cite[Lemma 2.2]{barrett1994finite}.
\end{proof}

\begin{lemma}\label{norm equivalent}
	For any vector $\mathbf a_i = (a_{i1}, \ldots, a_{id}) \in \mathbb{R}^d$, $i=1,\ldots,N$, 
	if $1 < r \leq s$, then
	\begin{align*}
  		\| \mathbf a \|_s \leq \| \mathbf a \|_r \leq (Nd)^{\frac{1}{r} - \frac{1}{s}} \| \mathbf a \|_s,	
	\end{align*}
	where $\mathbf a = (a_{ik}) \in \mathbb{R}^{N\times d}$ is a $N\times d$ matrix whose elements are given by $a_{ik}$, 
	and $\norm{\mathbf a }{\gamma}:= \left( \sumI{i} \sumK{k} a_{ik}^\gamma \right)^{\frac{1}{\gamma}}$.
\end{lemma}
\begin{proof}
	This lemma is very well-known. Therefore, we omit its proof, which can be obtained from \cite[Appendix C]{kim2020complete}.
\end{proof}

\begin{remark}\label{Rmk Norm-Equivalent}
	By Lemma \ref{norm equivalent}, it is clear that
	\begin{align}\label{norm_equiv_UC}
		\DC \LN{\mathbf a}{s}\leq \LN{\mathbf a}{r} \leq 
		\UC \LN{\mathbf a}{s}
	\end{align}
	for any $r > 1$, $s > 1$, and $\mathbf a \in \mathbb{R}^{N\times d}$.
	Here, $\UC$ and $\DC$ are defined by
	\begin{align*}
  	\UC:= \max\left\{ 1, (Nd)^{\frac{1}{r} - \frac{1}{s}} \right\}, 
  	\quad \text{ and }\quad
  	\DC:= \min\left\{ 1, (Nd)^{\frac{1}{r} - \frac{1}{s}} \right\}.
	\end{align*}
\end{remark}

\label{====================================}
\section{Norm type Rayleigh friction}\label{Section: Norm type}

In this section,
we discuss problems P1 and P2 by analyzing the nonlinear C-S model with norm-type Rayleigh friction:
\begin{align}
	&\frac{dx_{ik}}{dt} = v_{ik}    \label{Rayleigh friction (1)}\\
	&\frac{dv_{ik}}{dt} = \DpL v_{ik} + a v_{ik} \norm{v_i}{2}^{q-2}
	- b v_{ik} \norm{v_i}{2}^{r-2}, \label{Rayleigh friction (2)}
\end{align}
where $p>1$. 
In particular, we assume that $2\leq q < r$ to ensure that $\norm{\cdot}{2}^{q-2}$ and $\norm{\cdot}{2}^{r-2}$ are well-defined.
By the definitions of $v_{M,k}$ and $v_{m,k}$, \eqref{Rayleigh friction (2)} implies that for each $k\in\Kcal$, 
\begin{align*}
	v_{M,k}^\prime 
	\leq &
	b v_{M,k} \norm{v_M}{2}^{q-2} 
	\left(\frac{a}{b} - \norm{v_M}{2}^{r-q}  \right),
\end{align*}
and
\begin{align*}
	v_{m,k}^\prime 
	\geq &
	b v_{m,k} \norm{v_m}{2}^{q-2} 
	\left(\frac{a}{b} - \norm{v_m}{2}^{r-q}  \right).
\end{align*}
Thus, for a fixed $k\in\Kcal$, if $v_{M,k}(t)>\Cab$ on some time interval $I$, then $v_{M,k}$ is strictly decreasing on the interval $I$, and if $v_{m,k}(t)<-\Cab$ on some interval $I$, then $v_{m,k}(t)$ is strictly increasing on the interval $I$.
These properties indicate that for each $i\in\Ical$ and $k\in\Kcal$, $v_{i,k}$ is uniformly bounded with respect to $t$ by
\begin{align*}
	\min\left\{ - \Cab , v_{m,k}(0) \right\} 
	\leq v_{i,k} (t)
	\leq
	\max\left\{   \Cab , v_{M,k}(0) \right\}, \quad t\geq 0.
\end{align*}
Therefore, we have $|v_{i,k}(t)| \leq \max\left\{ \Cab , \max_{i\in\Ical, k \in\Kcal}|v_{i,k}(0)|\right\}$ for all $i\in\Ical$, $k\in\Kcal$, and $t\geq 0$.

We now discuss more detailed properties of $v_{i,k}$ under a non-negative (or non-positive) and non-zero initial velocity configuration.

\label{---}
\begin{lemma}\label{[Lm] Non-negativity Property of v for Rayleigh Friction}\label{Lm:abbr: Unif Bdd}
	If there exist $T\geq 0$ and $k\in\Kcal$ such that $v_{i,k}(T)$ is non-negative and non-zero \textup{(}that is, $v_{i,k}(T) \not\equiv 0$ with respect to $i$\textup{)}, then we have either
	\begin{align*}
  		v_{m,k}(t) >0, \quad t > T,
	\end{align*}
	or there exists $T'>T$ such that
	\begin{align*}
  		v_{m,k}(t) >0, ~ t\in (T, T'), \quad
  		\text{ and } \quad
  		v_{M,k}(t) = v_{m,k}(t) = 0, ~t\in[T', \infty),
	\end{align*}
\end{lemma}
\begin{proof}
	For the given $k\in\Kcal$, we first assume $v_{m,k}(T) = 0$.
	Since $v_{i,k}$ is non-zero at $t=T$, we have $0=v_{m,k}(T) < v_{M,k}(T)$, implying that $v_{m,k}^\prime (T)=\DpL v_{m,k} (T)>0$.
	Hence, there exists $\delta>0$ such that $v_{m,k}(T+\delta)>0$. 
	Consequently, without loss of generality, we may assume that $v_{m,k}(T)>0$.
	If there exists a minimum $T'>T$ such that $v_{m,k} (T') =0$, then since $v_{i,k}\in C^1([0,\infty))$, $v_{m,k}^\prime (T') \leq 0$.
	Hence, we obtain
	\begin{align*}
  		0 \geq v_{m,k}^\prime (T') = \DpL v_{m,k}(T') \geq 0,
	\end{align*}
	implying $v_{i,k}(t) =0$ for all $i\in\Ical$ and $t\geq T'$.
	If not (i.e. $v_{m,k}(t) >0$ for all $t\geq T$), then it is trivial that $v_{i,k}(t) >0$ for all $i\in\Ical$ and $t\geq T$. 
	Thus, we have the desired result.
\end{proof}

We note that if $x_{i,k}$ and $v_{i,k}$ are solutions to \eqref{Rayleigh friction (1)}-\eqref{Rayleigh friction (2)}, then $-x_{i,k}$ and $-v_{i,k}$ are also a solutions to \eqref{Rayleigh friction (1)}-\eqref{Rayleigh friction (2)}. 
Therefore, from this fact, we obtain the next lemma.

\begin{lemma}\label{[Lm] Non-positivity Property of v for Rayleigh Friction}\label{Lm:abbr: Unif Bdd(2)}
	If there exist $T\geq 0$ and $k\in\Kcal$ such that $v_{i,k}(T)$ is non-positive and non-zero \textup{(}that is, $v_{i,k}(T) \not\equiv 0$ with respect to $i$\textup{)}, then either
	\begin{align*}
  		v_{M,k}(t) < 0, \quad t > T,
	\end{align*}
	or there exists $T'>T$ such that
	\begin{align*}
  		v_{M,k}(t) < 0, ~ t\in (T, T'), \quad
  		\text{ and } \quad
  		v_{M,k}(t) = v_{m,k}(t) = 0, ~t\in[T', \infty),
	\end{align*}
\end{lemma}

\begin{lemma}\label{(Lm) norm of v_M < Cab}
	Let $v_{i,k}(0)$ be non-negative and non-zero for all $i\in\Ical$ and $k\in\Kcal$.
	If $\norm{v_M(T)}{2}\leq \Cab$ for some $T>0$, then $\norm{v_M(t)}{2}\leq \Cab$ for all $t\geq T$.
\end{lemma}
\begin{proof}
	We provide a proof by contradiction.
	Suppose that there exists $\tau_0>T$ such that $\norm{v_M(\tau_0)}{2} > \Cab$.
	Since $\norm{v_M(T)}{2}\leq \Cab$, there exists $(t_0, t_1)\subset (T, \tau_0)$ such that 
	\begin{align}\label{Lm: Pf 010}
  		\frac{d}{dt}\norm{v_M(t)}{2} >0, \quad \text{and} \quad \norm{v_M(t)}{2}>\Cab
	\end{align}
	for all $t\in (t_0, t_1)$.
	From this, we obtain
	\begin{align}\label{Lm: Pf 020}
  		0
  		<&
  		\frac{1}{2} \frac{d}{dt}\norm{v_M(t)}{2}^2 \\
  		= & 
  		\sumK{k} v_{M,k}(t) v_{M,k}^\prime (t)\nonumber \\
  		= &
  		\sumK{k} v_{M,k}(t) 
  		\left( \DpL v_{M,k}(t) + a v_{M,k}(t) \norm{v_M(t)}{2}^{q-2}
		- b v_{M,k}(t) \norm{v_M(t)}{2}^{r-2} \right), \nonumber
		\quad 
		t\in (t_0 , t_1).
	\end{align}
	By contrast, since $v_{M,k}(t) \geq 0$ for all $k\in\Kcal$ and $t\geq 0$ by Lemma \ref{Lm:abbr: Unif Bdd}, we have $\sumK{k} v_{M,k}(t) \DpL v_{M,k}(t) \leq 0$ for all $t\geq0$, implying that
	\begin{align*}
  		\sumK{k} v_{M,k}(t) 
  		\left( \DpL v_{M,k}(t) + a v_{M,k}(t) \norm{v_M(t)}{2}^{q-2}
		- b v_{M,k}(t) \norm{v_M(t)}{2}^{r-2} \right)
		\leq 
		b \norm{v_M(t)}{2}^q \left( \frac{a}{b} - \norm{v_M(t)}{2}^{r-q}  \right)
		\leq 0,
	\end{align*}
	for all $t\geq 0$.
	The last inequality follows from $\norm{v_M(t)}{2}>\Cab$ in \eqref{Lm: Pf 010} and $b>0$.
	This contradicts \eqref{Lm: Pf 020}.
\end{proof}

\begin{lemma}\label{(Lm) norm of v_m > Cab}
	Let $v_{i,k}(0)$ be non-negative and non-zero for all $i\in\Ical$ and $k\in\Kcal$.
	If $\norm{v_m(T)}{2}\geq \Cab$ for some $T>0$, then $\norm{v_m(t)}{2}\geq \Cab$ for all $t\geq T$.
\end{lemma}
\begin{proof}
	This lemma can be proved by the same argument as in the proof of Lemma \ref{(Lm) norm of v_M < Cab} above.
	Hence, we omit the details.
\end{proof}

Note that by Lemma \ref{(Lm) norm of v_M < Cab} and Lemma \ref{(Lm) norm of v_m > Cab}, we obtain the following trichotomy-type result:

\begin{proposition}\label{(Prop) Trichotomy for Rayleigh friction}
	Let $v_{i,k}(0)$ be non-negative and non-zero for all $i\in\Ical$ and $k\in\Kcal$.
	Then, there exists $T>0$ such that one of the following is true:
	\begin{itemize}
		\item[\textup{(i)}] $\norm{v_M(t)}{2}\leq \Cab$ for all $t\geq T$.
		\item[\textup{(ii)}] $\norm{v_m(t)}{2}\geq \Cab$ for all $t\geq T$.
		\item[\textup{(iii)}] $\norm{v_m(t)}{2} < \Cab < \norm{v_M(t)}{2}$ for all $t\geq T$.
	\end{itemize}
\end{proposition}
\begin{proof}
	It is trivial from Lemma \ref{(Lm) norm of v_M < Cab} and Lemma \ref{(Lm) norm of v_m > Cab}.
\end{proof}

For each $i\in\Ical$,
we now discuss the convergence of $\norm{v_i}{2}$ for the non-negative and non-zero initial velocity configuration and present estimates for $\norm{v_i}{2}$.

\begin{theorem}\label{[Thm] Norm Convergence to Cab}
	Let $v_{i,k}(0)$ be non-negative and non-zero for all $i\in\Ical$ and $k\in\Kcal$.
	Then, we have
	\begin{align*}
  		\lim_{t\to\infty} \norm{v_i(t)}{2} = \Cab
	\end{align*}
	for all $i\in\Ical$.
\end{theorem}
\begin{proof}
	{\bf Step 1.} We first prove that $\lim_{t\to\infty} \norm{v_M (t)}{2}\leq \Cab$.
	If there exists $T\geq 0$ such that $\norm{v_M(T)}{2}\leq \Cab$, then by Lemma \ref{(Lm) norm of v_M < Cab}, $\norm{v_M(t)}{2}\leq \Cab$ for all $t\geq T$.
	Thus, in this case, $\lim_{t\to\infty} \norm{v_M (t)}{2}\leq \Cab$.
	We now consider the opposite case, that is, $\norm{v_M(t)}{2}> \Cab$ for all $t\geq 0$. Since $v_{i,k}(0)$ is non-negative and non-zero for all $i\in\Ical$ and $k\in\Kcal$, it follows from Lemma \ref{Lm:abbr: Unif Bdd} that $v_{M,k}(t) \geq 0$ for all $k\in\Kcal$ and $t\geq 0$. 
	Therefore, we have
	\begin{align}\label{Prop: Pf 010}
  		\frac{1}{2} \frac{d}{dt} \norm{v_M(t)}{2}^2
  		= &
  		\sumK{k} v_{M,k}(t) 
  		\left( \DpL v_{M,k}(t) + a v_{M,k}(t) \norm{v_M(t)}{2}^{q-2}
		- b v_{M,k}(t) \norm{v_M(t)}{2}^{r-2} \right)\nonumber \\
  		\leq &
  		b \norm{v_M(t)}{2}^{q} \left( \frac{a}{b} - \norm{v_M(t)}{2}^{r-q} \right)\\
  		<& 
  		0, \nonumber
	\end{align}
	for all $t\geq0$, that is, $\norm{v_M(t)}{2}$ is strictly decreasing on $[0,\infty)$.
	Thus, there exists $\alpha\geq\Cab$ such that $\lim_{t\to\infty} \norm{v_M (t)}{2}=\alpha$. We claim $\alpha = \Cab$. To show this, suppose, contrary to our claim, that $\alpha > \Cab$.
	Then, \eqref{Prop: Pf 010} implies
	\begin{align*}
  		\frac{1}{2} \frac{d}{dt} \norm{v_M(t)}{2}^2
  		\leq
  		b \norm{v_M(t)}{2}^{q} \left( \frac{a}{b} - \norm{v_M(t)}{2}^{r-q} \right)
  		\leq
  		b \alpha^q \left( \frac{a}{b} - \alpha^{r-q} \right)<0.
	\end{align*}
	By integrating it from 0 to $t$, we deduce
	\begin{align*}
  		\norm{v_M(t)}{2}^2
  		< 
  		\norm{v_M(0)}{2}^2
  		- 2b \alpha^{q-1} \left( \alpha^{r-q} - \frac{a}{b} \right) t
  		\to -\infty \text{ as } t \to \infty,
	\end{align*}
	which contradicts the fact that $\norm{v_M(t)}{2}^2\geq 0$.
	Therefore, we have $\lim_{t\to\infty} \norm{v_M (t)}{2}\leq \Cab$.
	
	{\bf Step 2.} 
	We now show that $\lim_{t\to\infty} \norm{v_m (t)}{2}\geq \Cab$.
	The proof of this statement is similar to the above proof and is discussed briefly.
	If there exists $T\geq 0$ such that $\norm{v_m (T)}{2} \geq \Cab$, then it follows from Lemma \ref{(Lm) norm of v_m > Cab} that $\lim_{t\to\infty} \norm{v_m (t)}{2}\geq \Cab$.
	We now assume that $\norm{v_m (t)}{2} < \Cab$ for all $t\geq0$.
	Then, as above, we can show that $\norm{v_m (t) }{2}$ is non-decreasing on $[0,\infty)$.
	Therefore, $\norm{v_m }{2}$ converges to $\alpha\in[0,\Cab]$ as $t\to\infty$. 
	If $\alpha=0$, then since $\norm{v_m }{2}$ is non-decreasing on $[0,\infty)$, $\norm{v_m (t) }{2} = 0$ for all $t\geq0$. However, by Lemma \ref{Lm:abbr: Unif Bdd}, there exists some $T>0$ such that $\norm{v_m (t) }{2} > 0$ for $t\in(0,T)$, which is a contradiction. Thus, $\alpha \in \left(0,\Cab \right)$. 
	Finally, by applying an argument similar to the above, we deduce that $\alpha = \Cab$.
	
	Therefore, from Steps 1 and 2, we obtain $\Cab \leq \lim_{t\to\infty} \norm{v_m (t)}{2} \leq \lim_{t\to\infty} \norm{v_M (t)}{2}\leq \Cab$.
	Thus, we have the desired result.
\end{proof}

\label{---}
\begin{proposition}\label{[Prop] Estimate for v_M}
	Let $v_{i,k}(0)$ be non-negative and non-zero for all $i\in\Ical$ and $k\in\Kcal$.
	If there exists $T\geq0$ such that $\norm{v_M(t)}{2}> \Cab$ for all $t\geq T$, then we have the decay estimate
	\begin{align*}
  		\norm{v_M(t)}{2}^2 - \left(\frac{a}{b}\right)^{\frac{2}{r-q}} 
  		\leq
  		\left( \norm{v_M(T)}{2}^2 - \left(\frac{a}{b}\right)^{\frac{2}{r-q}} \right) 
  		e^{- b \xi_0 C^q (r-q) (t- T)}, \quad t\geq T,
	\end{align*}
	where $\xi_0:= \left(\frac{a}{b}\right)^{\frac{r-q-2}{r-q}}$ if $r-q-2\geq0$ and $\xi_0:=  \norm{v_M(T)}{2}^{r-q-2}$ if $r-q-2<0$.
\end{proposition}
\begin{proof}
	We denote $\Cab$ briefly by $C$. 
	Since $v_{m,k} \DpL v_{m,k} \leq 0$ for all $k\in\Kcal$, we have
	\begin{align*}
  		\frac{1}{2} \frac{d}{dt} \left( \norm{v_M}{2}^2 - C^2\right)
  		= & 
  		\Inpro{v_M}{\DpL v_M + a v_M \norm{v_M}{2}^{q-2} - b v_M \norm{v_M}{2}^{r-2}}\\
  		\leq &
  		\Inpro{v_M}{a v_M \norm{v_M}{2}^{q-2} - b v_M \norm{v_M}{2}^{r-2}}\\
  		= & 
  		 - b \norm{v_M}{2}^q \left(\norm{v_M}{2}^{r-q} - C^{r-q} \right).
	\end{align*}
	By the mean value theorem, for each $t\geq T$, there exists $\xi(t)\in (C^2, \norm{v_M(t)}{2}^2)$ such that 
	\begin{align*}
  		\frac{\left( \norm{v_M}{2}^2 \right)^{\frac{r-q}{2}} - \left(C^2\right)^{\frac{r-q}{2}}}{\norm{v_M}{2}^2 - C^2} 
  		=
  		\left( \frac{r-q}{2} \right) \xi^{\frac{r-q-2}{2}}(t).
	\end{align*}
	Since $\norm{v_M}{2}$ is strictly decreasing on $[T,\infty)$, 
	$\norm{v_M(t)}{2}\leq \norm{v_M(T)}{2}$ for all $t\geq T$.
	If $r-q-2\geq0$, then $\xi^{\frac{r-q-2}{2}}(t)\geq C^{r-q-2}$, and
	if $r-q-2<0$, then $\xi^{\frac{r-q-2}{2}}(t)\leq \norm{v_M(T)}{2}^{r-q-2}$; therefore,
	we have
	\begin{align*}
  		\frac{1}{2} \frac{d}{dt} \left( \norm{v_M(t)}{2}^2 - C^2\right)
  		\leq &
  		 - b \xi_0 \left( \frac{r-q}{2} \right)  \norm{v_M(t)}{2}^q 
  		 \left(\norm{v_M(t)}{2}^2 - C^2 \right)\\
  		 \leq &
  		 - b \xi_0 C^q \left( \frac{r-q}{2} \right)  
  		 \left(\norm{v_M(t)}{2}^2 - C^2 \right),
	\end{align*}
	for $t\geq T$. Separation and integration on both sides give
	\begin{align*}
  		\norm{v_M(t)}{2}^2 - C^2 
  		\leq & 
  		\left( \norm{v_M(T)}{2}^2 - C^2 \right) 
  		e^{- b \xi_0 C^q \left( r-q  \right) (t- T)}, \quad t\geq T.
	\end{align*}
\end{proof}

\label{---}
\begin{proposition}\label{[Prop] Estimate for v_m}
	Let $v_{i,k}(0)$ be non-negative and non-zero for all $i\in\Ical$ and $k\in\Kcal$.
	If there exists $T\geq0$ such that $\norm{v_m(t)}{2} < \Cab$ for all $t\geq T$, the following decay estimate is obtained:
	\begin{align*}
  		\left(\frac{a}{b}\right)^{\frac{2}{r-q}}  - \norm{v_m(t)}{2}^2 
  		\leq
  		\left( \left(\frac{a}{b}\right)^{\frac{2}{r-q}} - \norm{v_M(T)}{2}^2 \right) 
  		e^{- b \xi_0 \norm{v_m(T)}{2}^q \left( r-q  \right) (t- T)}, \quad t\geq T,
	\end{align*}
	where $\xi_0:= \norm{v_m(T)}{2}^{r-q-2} $ if $r-q-2\geq0$ and $\xi_0:=  \left(\frac{a}{b}\right)^{\frac{r-q-2}{r-q}}$ if $r-q-2<0$.
\end{proposition}
\begin{proof}
	This proposition is proved by the same argument as the proof of Proposition \ref{[Prop] Estimate for v_M}. Thus, details are left to the reader.
\end{proof}

From this point, we consider an arbitrary initial velocity configuration and discuss the finite flocking time.
In particular, we assume that $p<2$, that is, $1<p<2$ and $2\leq q <r$. 
We assume $1<p<2$ because the discrete $p$-Laplacian induces consensus in finite time when $1<p<2$. For more details, see \cite{kim2020complete}.
We note that if there exists a finite flocking time, then since all agents have the same velocity after some time, it is trivial that the system \eqref{Rayleigh friction (1)}-\eqref{Rayleigh friction (2)} forms the flocking.
Therefore, we only deal with the finite flocking time.

\begin{theorem}\label{(Thm) Consensus for Rayleigh friction}
	For an arbitrary initial configuration and sufficiently small $a$, if $\psi$ satisfies 
	\begin{align}\label{Thm: Eq 010}
  		4 C_m \left( 1 - \frac{p}{2} \right) \int_0^\infty 
  		\min_{i,j\in\Ical} \CSweightT dt 
  		> \norm{v_M (0) - v_m (0)}{2}^{2-p},
	\end{align}
	then the system \eqref{Rayleigh friction (1)}-\eqref{Rayleigh friction (2)} has a finite flocking time.
\end{theorem}
\begin{proof}
	For simplicity, we write $\pmin (t)$ instead of $\min_{i,j\in\Ical} \CSweightT$.
	Since
	\begin{align*}
  		\DpL v_{M,k} - \DpL v_{m,k} 
  		\leq 
  		- 2 \pmin ( v_{M,k} - v_{m,k})^{p-1},
	\end{align*}
	and 
	\begin{align*}
  		- b v_{M,k} \norm{v_M}{2}^{r-2} + b v_{m,k} \norm{v_m}{2}^{r-2} \leq 0,
	\end{align*}
	for $k\in\Kcal$, we have
	\begin{align*}
  		\frac{1}{2}\left( \norm{v_M - v_m }{2}^2\right)^\prime
  		= &
  		\Inpro{v_M - v_m}{v_M^\prime - v_m^\prime}\\
  		\leq &
  		-2 \pmin \norm{v_M - v_m}{p}^p 
  		+ a \Inpro{v_M - v_m}{v_M \norm{v_M}{2}^{q-2} - v_m \norm{v_m}{2}^{q-2} }\\
  		\leq &
  		-2 \pmin \norm{v_M - v_m}{p}^p 
  		+ a \norm{v_M - v_m}{2} \norm{v_M \norm{v_M}{2}^{q-2} - v_m \norm{v_m}{2}^{q-2}}{2}.
	\end{align*}
	Then, by applying Lemma \ref{L-inequalities} with $\delta=0$, we obtain the Bernoulli-type inequality:
	\begin{align*}
  		\frac{1}{2}\left( \norm{v_M - v_m }{2}^2\right)^\prime
  		\leq &
  		-2 \pmin \norm{v_M - v_m}{p}^p 
  		+ a C_1 \norm{v_M - v_m}{2}^2 
  		\left( \norm{v_M}{2} + \norm{v_m}{2}\right)^{q-2}\\
  		\leq &
  		-2 C_m \pmin \norm{v_M - v_m}{2}^p 
  		+ a C_1 M^{q-2} \norm{v_M - v_m}{2}^2,
	\end{align*}
	for some $M>0$.
	Here, the constant $M$ is obtained by the uniform boundedness of $v_{i,k}$.
	By solving it, we deduce
	\begin{align*}
  		\norm{v_M(t) - v_m (t)}{2}^{2-p} e^{-2 a \left(1- \frac{p}{2}\right)  M^{q-2} t} 
  		\leq
  		\norm{v_M(0) - v_m (0)}{2}^{2-p}
  		- 4 C_m \left(1- \frac{p}{2}\right) 
  		\int_0^t \pmin(s) 
  		 e^{-2 a \left(1- \frac{p}{2}\right)  M^{q-2} s} ds.
	\end{align*}
	Since we assume that 
	\begin{align*}
		\norm{v_M (0) - v_m (0)}{2}^{2-p}
		<
  		4 C_m \left( 1 - \frac{p}{2} \right) \int_0^\infty \pmin(s) ds,
	\end{align*}
	and it is clear that 
	\begin{align*}
  		\int_0^\infty \pmin(s) e^{-2 a \left(1- \frac{p}{2}\right)  M^{q-2} s} ds \to \int_0^\infty \pmin(s)  ds \text{ as } a \to 0,
	\end{align*}
	there exists a small enough $\epsilon_0>0$ such that 
	\begin{align}\label{Thm: Eq 020}
		\norm{v_M(0) - v_m (0)}{2}^{2-p}
		<
  		4 C_m \left(1- \frac{p}{2}\right) 
  		\int_0^\infty \pmin(s) e^{-2 a \left(1- \frac{p}{2}\right)  M^{q-2} s} ds,
  		\quad 0<a<\epsilon_0.
	\end{align}
	Therefore, for a sufficiently small $a$, there exists $T>0$ such that
	\begin{align*}
		\norm{v_M(0) - v_m (0)}{2}^{2-p}
		=
  		4 C_m \left(1- \frac{p}{2}\right) 
  		\int_0^T \pmin(s) e^{-2 a \left(1- \frac{p}{2}\right)  M^{q-2} s} ds.
	\end{align*}
	Hence, the system \eqref{Rayleigh friction (1)}-\eqref{Rayleigh friction (2)} has a finite flocking time $T^* \leq T$.
\end{proof}

\begin{remark}
	(i) The important point in the above proof is solving the ordinary differential equation $y\leq -ay^{\frac{p}{2}} + by$. In particular, if $p\geq 2$, then we can only technically know that $y$ decays exponentially or polynomially but not whether it flocks in finite time.
	
	(ii) 
	We should also note that the parameter $a$ must be a very small positive value.
	The reason for this is due to \eqref{Thm: Eq 020} in the proof above.
	Since the term $\pmin$ in \eqref{Thm: Eq 020} depends on the change in distance of the agents over time, the exact information about how small a should be is not known from the above proof.
	The condition that $a$ is small enough will be assumed a few more times in Section 3 and Section 4 of this thesis. However, for the same reasons as above, we do not propose an estimate for $a$ in this paper.
\end{remark}

We now discuss the convergence of $\norm{v_i}{2}$, $i\in\Ical$.
In fact, for an arbitrary initial velocity configuration, we have the three cases that for each $k\in\Kcal$,
\begin{enumerate}
	\item[(i)] there exists $T\geq 0$ such that $v_{i,k}(T)$ is non-negative,
	\item[(ii)] there exists $T\geq 0$ such that $v_{i,k}(T)$ is non-positive,
	\item[(iii)] $v_{m,k}(t) < 0 < v_{M,k}(t)$, for all $t\geq0$.
\end{enumerate}
For cases (i) and (ii), if $v_{i,k} (T) \equiv 0$ with respect to $i$, then $v_{i,k}(t) = 0$ for all $i$, and $t\geq T$.
If $v_{i,k}(T)$ is non-zero and non-negative (or non-positive), then by Lemma \ref{Lm:abbr: Unif Bdd} (or Lemma \ref{Lm:abbr: Unif Bdd(2)}), we have $v_{i,k}(t)>0$ (or $<0$) for all $t\geq T$ or there exists $T'>T$ such that $v_{i,k} (t) = 0$ for all $i$ and $t\geq T'$.
Moreover, if we consider sufficiently small $a$ and $\psi$ satisfying \eqref{Thm: Eq 010}, then by Theorem \ref{(Thm) Consensus for Rayleigh friction}, case (iii) does not occur.
Therefore, for each $k\in\Kcal$, there exists $T\geq 0$ such that either $|v_{i,k}(t)|>0$ or $v_{i,k}(t) =0$ for all $i\in\Ical$ and $t\geq T$.
Hence, only two cases can occur for $v_{i,k}$: there exists $T\geq 0$ such that
\begin{enumerate}
	\item[($\rm{i}^\prime$)] $v_{i,k}(t)=0$ for all $i$, $k$, and $t\geq T$,
	\item[($\rm{ii}^\prime$)] there exists $k\in\Kcal$ such that $|v_{i,k}(t)|>0$ for all $i$, and $t>T$.
\end{enumerate}
For ($\rm{i}^\prime$), it is clear that $\lim_{t\to \infty} \norm{v_i(t)}{2}=0$ for all $i$. For ($\rm{ii}^\prime$), if we assume all the conditions in Theorem \ref{(Thm) Consensus for Rayleigh friction}, then there exists $T'>0$ such that $\DpL v_{i,k} (t) =0$ for all $i$ and $t\geq T'$. 
We have
\begin{align*}
  	\frac{1}{2}\frac{d}{dt}\norm{v_i(t)}{2}^2 
  	=&
  	\Inpro{v_i(t)}{a v_i(t) \norm{v_i(t)}{2}^{q-2} - b v_i(t) \norm{v_i(t)}{2}^{r-2}}\\
  	=&
  	b \norm{v_i(t)}{2}^{q}\left( \frac{a}{b} - \norm{v_i(t)}{2}^{r-q}\right),
  	\quad t\geq T'.
\end{align*}
Since $\norm{v_i(t)}{2}>0$ for all $t>0$, if $\norm{v_i(t)}{2}<\Cab$ on some time interval $I$, then $\norm{v_i(t)}{2}$ is strictly increasing on $I$, and if $\norm{v_i(t)}{2}>\Cab$, then $\norm{v_i(t)}{2}$ is strictly decreasing on $I$.
Hence, $\lim_{t\to\infty}\norm{v_i(t)}{2}=\Cab$.
Therefore, we have the following result:

\label{---}
\begin{theorem}\label{[Thm] Dychotomy Convergence}
	For an arbitrary initial configuration, suppose $a$ is small enough and $\psi$ satisfies 
	\begin{align}\label{[Thm] DC Eq 010}
  		4 C_m \left( 1 - \frac{p}{2} \right) \int_0^\infty \min_{i,j\in\Ical} \CSweightT dt > \norm{v_M (0) - v_m (0)}{2}^{2-p}.
	\end{align}
	Then, either $\lim_{t\to\infty}\norm{v_i(t)}{2}=\Cab$ or there exists $T\geq0$ such that  $\norm{v_i(t)}{2}=0$ for all $t\geq T$.
\end{theorem}

We are interested in the non-zero convergence to the solution $v_{i,k}$ of the system in \eqref{Rayleigh friction (1)}-\eqref{Rayleigh friction (2)}.
The next proposition presents a condition for the initial velocity configuration to guarantee the non-zero convergence of $\norm{v_i}{2}$.
By Theorem \ref{[Thm] Norm Convergence to Cab}, 
if there exists $T\geq0$ such that $v_{i,k}(0)$ is non-negative (or non-positive) and non-zero for all $i$ and $k$, then $\norm{v_i(t)}{2}$ converges to $\Cab$ as $t\to\infty$. Hence, we now only discuss the case where $v_{m,k}(0) < 0 < v_{M,k}(0)$ for some $k\in\Kcal$.

\begin{proposition}\label{[Prop] Positive Convergence}
	Suppose $a$ is small enough and $\psi$ satisfies 
	\begin{align*}
  		4 C_m \left( 1 - \frac{p}{2} \right) \int_0^\infty \min_{i,j\in\Ical} \CSweightT dt > \norm{v_M (0) - v_m (0)}{2}^{2-p}.
	\end{align*}
	For all $k\in\Kcal$ satisfying 
	\begin{align*}
  		v_{m,k}(0)<0<v_{M,k}(0),
	\end{align*}
	if $|v_{m,k}(0)|$ is small enough, then $\norm{v_i(t)}{2}\to\Cab$ as $t\to\infty$.
\end{proposition}
\begin{proof}
	We present a proof by contradiction. 
	Suppose that $\norm{v_i(t)}{2}\not\to\Cab$ as $t\to\infty$.
	By Theorems \ref{[Thm] Norm Convergence to Cab} and \ref{(Thm) Consensus for Rayleigh friction}, there exists a minimum $T>0$ such that $\norm{v_i(t)}{2}=0$ for $t\geq T$. 
	Moreover, there exists $k\in\Kcal$ such that $v_{m,k}(t)<0<v_{M,k}(t)$ for $t\in[0,T)$.
	We calculate
	\begin{align*}
  		\frac{d}{dt}v_{m,k}(t) 
  		\geq & 
  		\pmin ( v_{M,k}(t) - v_{m,k}(t))^{p-1} 
  			+ a v_{m,k}(t) \norm{v_m(t)}{2}^{q-2} - b v_{m,k}(t) \norm{v_m(t)}{2}^{r-2}\\
  		\geq &
  		\pmin ( - v_{m,k}(t))^{p-1} + a v_{m,k}(t) \norm{v_m (t)}{2}^{q-2},
	\end{align*}
	for all $t\in[0,T)$.
	Since $q\geq2$ and $v_{i,k}$ is uniformly bounded, 
	we have
	\begin{align*}
  		\frac{d}{dt}\left( - v_{m,k}(t) \right)
  		\leq & 
  		- \pmin ( - v_{m,k})^{p-1} + a M^{q-2} \left( - v_{m,k} \right).
	\end{align*}
	This Bernoulli-type inequality gives us
	\begin{align*}
  		(- v_{m,k}(t))^{2-p} e^{ -aM^{q-2}(2-p) t }
  		\leq
  		(- v_{m,k}(0))^{2-p} -(2-p)\int_0^t \pmin(s) e^{ -aM^{q-2}(2-p) s }ds
	\end{align*}
	for all $t\in[0,T)$.
	Since $\int_0^\infty \pmin(s) e^{ -aM^{q-2}(2-p) s }ds \to \int_0^\infty \pmin(s)ds$ as $a\to0$, there exists a sufficiently small $a$ such that
	\begin{align*}
  		\int_0^\infty \pmin(s) ds
  		\geq
  		\int_0^\infty \pmin(s) e^{ -aM^{q-2}(2-p) s }ds
  		> 
  		\frac{1}{2C_m (2-p)} \norm{v_M (0) - v_m (0)}{2}^{2-p}.
	\end{align*}
	Since it is clear that
	\begin{align*}
  		\frac{1}{2C_m (2-p)} \norm{v_M (0) - v_m (0)}{2}^{2-p}
  		>
  		\frac{1}{2(2-p)} |v_{m,k}(0)|^{2-p},
	\end{align*}
	there exists $T^*>0$ such that 
	\begin{align*}
  		(- v_{m,k}(0))^{2-p} = 
  		2(2-p)\int_0^{T^*} \pmin(s) e^{ -aM^{q-2}(2-p) s }ds.
	\end{align*}
	Thus, there exists $0<T_0\leq T^*$ such that $v_{m,k}(t)=0$ for all $t \geq T_0$.
	We choose $|v_{m,k}(0)|$ to be sufficiently small such that $T_0<T$,
	which contradicts the assumption $v_{m,k}<0$ on $[0, T)$.
\end{proof}

\begin{remark}
	The condition that $|v_{m,k}(0)|$ must be very small in the above proposition was proposed to derive a contradiction in a proof using the induction method. 
	Therefore, an exact estimate of how small $|v_{m,k}(0)|$ must be is not known. However, the authors conjecture that it should be small enough that the initial velocity configuration is not origin symmetric, but this needs further study. The basis for this conjecture can be seen in examples in Section 6.
\end{remark}


\label{====================================}
\section{Vector-type Rayleigh friction}\label{Section: Abs type}

As observed in the previous section, 
for each $i\in\Ical$,
we could only analyze the convergence of $\norm{v_i}{2}$ for systems \eqref{Rayleigh friction (1)}-\eqref{Rayleigh friction (2)} (that is, the speed of $i$-th agent), and could not obtain information about the convergence of each $v_i$ (i.e., the velocity of $i$-th agent) owing to technical limitations.
In this section, we propose the following nonlinear C-S model with (vector-type) Rayleigh friction to analyze or control the convergence of $v_i$:
\begin{align}
	&\frac{dx_{ik}}{dt} = v_{ik}    \label{Main_eq(1)}\\
	&\frac{dv_{ik}}{dt} = \DpL v_{ik} + a_k \varphi_q(v_{ik}) 
	- b_k \varphi_r(v_{ik}), \label{Main_eq(2)}
\end{align}
where $p>1$, $1<q<r$, and $\varphi_\gamma(s) := |s|^{\gamma-2}s$, $s\in\mathbb{R}$, $\gamma>1$.
We note that unlike in Section 3, $q>1$ is assumed for well-definedness by the definition of $\varphi_\gamma$, and when we discuss the finite flocking time, we assume $1<p<2$ and $p<q<r$.

We first discuss the uniform boundedness of $v_{i,k}$ for $p>1$ and $q<r$.
Since $\DpL v_{M,k} \leq 0$ and $\DpL v_{m,k} \geq 0$, it follows from \eqref{Main_eq(2)} that
\begin{align}
	\frac{d}{dt} v_{M,k} 
		\leq & 
		b_k \varphi_q(v_{M,k}) \left( \frac{a_k}{b_k} - |\varphi_{r-q+1} (v_{M,k})| \right), \label{unifbd ineq}
\end{align}
and similarly,
\begin{align}\label{Unifbd Ineq 2}
	\frac{d}{dt} v_{m,k}
	\geq  
	b_k \varphi_q(v_{m,k}) \left( \frac{a_k}{b_k} - |\varphi_{r-q+1} (v_{m,k})| \right),
\end{align}
for $k\in\Kcal$.
Therefore, if $v_{M,k} >\Cabk$ on some time interval $I$, then $v_{M,k}$ is strictly decreasing on $I$. 
If $v_{m,k} < -\Cabk$ on some time interval $I$, then $v_{m,k}$ is strictly increasing on $I$.
Since $v_{i,k}$ is continuous, we conclude that $v_{i,k}$ is uniformly bounded with respect to $t$ for all $i$ and $k$, and we present the constant of uniform boundedness using the next two results.

\begin{lemma}\label{Lm Upper Bound}
	For a fixed $k\in\Kcal$, if $v_{M,k}(T) \leq \Cabk$ for some $T\geq0$, then
	\begin{align*}
		v_{M,k}(t) \leq \Cabk, \quad t\geq T.
	\end{align*}
	If $v_{m,k}(T) \geq \Cabk$ for some $T\geq0$, then 
	\begin{align*}
		v_{m,k}(t) \geq \Cabk, \quad t\geq T.
	\end{align*}
\end{lemma}
\begin{proof}
	We present a proof by contradiction. Suppose that there exists $T'>T$ such that $v_{M,k}(T')>\Cabk$. 
	Then, there exists $(t_1, t_2)\subset (T,T')$ such that $v_{M,k}^\prime(t) >0$ and $v_{M,k}(t) > \Cabk$ for $t\in (t_1,t_2)$.
	Then, \eqref{unifbd ineq} implies
	\begin{align*}
		0 < v_{M,k}^\prime (t) \leq b_k \varphi_q(v_{M,k}(t)) \left( \frac{a_k}{b_k} - \left|\varphi_{r-q+1} (v_{M,k}(t))\right| \right)<0,\quad t \in (t_1, t_2),
	\end{align*}
	which is a contradiction.	
	Thus, we have the desired result. 
	Finally, by applying a method similar to the above proof to \eqref{Unifbd Ineq 2}, we can establish $v_{m,k}(t) \geq \Cabk$, $t\geq T$.	
\end{proof}

\begin{proposition}\label{Prop Uniformly Bounded for v}
	For $p>1$ and $q<r$, and $k\in\Kcal$, the solution $v_{i,k}$ satisfies 
	\begin{align*}
  		B_{m,k}:=\min \left\{v_{m,k}(0) , -\Cabk \right\} 
  		\leq v_{i,k} (t)
  		\leq 
  		\max \left\{v_{M,k}(0) , \Cabk \right\} =:B_{M,k}
	\end{align*}
	for all $i\in\Ical$ and $t\geq 0$.
\end{proposition}
\begin{proof}
	If $v_{M,k}(0) \leq \Cabk$, then by Lemma \ref{Lm Upper Bound}, $v_{M,k}(t) \leq \Cabk$ for all $t\geq 0$. 
	If $v_{M,k}(0) > \Cabk$, then there exists $T>0$ ($T$ may be infinite) such that $v_{M,k}$ is strictly decreasing on $[0,T)$.
	Therefore, we have 
	\begin{align*}
  		v_{i,k}(t) \leq v_{M,k}(t) 
  		\leq \max \left\{v_{M,k}(0) , \Cabk \right\}.
	\end{align*}
	Similarly, we also have
	\begin{align*}
  		v_{i,k}(t) \geq v_{m,k}(t) 
  		\geq \min \left\{v_{m,k}(0) , -\Cabk \right\}.
	\end{align*}
\end{proof}

From Lemma \ref{Lm Upper Bound}, we can obtain the following trichotomy-type result.

\begin{proposition}\label{Prop Trichotomy}
	For a fixed $k\in\Kcal$, there exists $T\geq 0$ that allows only one of the following to be true:
	\begin{itemize}
	\item[\textup{(i)}] $v_{M,k}(t) \leq \Cabk$, $t\geq T$,
	\item[\textup{(ii)}] $v_{m,k}(t) \geq \Cabk$, $t\geq T$,
	\item[\textup{(iii)}] $v_{m,k}(t) < \Cabk < v_{M,k}(t)$, $t\geq T$,
	\end{itemize}
	for an arbitrary initial configuration.
\end{proposition}
\begin{proof}
	It is trivial by Lemma \ref{Lm Upper Bound}.
\end{proof}

So far, we have discussed the convergence of $v_{i,k}$ as $t\to\infty$.
If we assume a positive condition on the initial velocity configuration, we obtain the following positivity property.

\begin{lemma}\label{Lm Non-negativity}
	For a fixed $k\in\Kcal$, 
	\begin{itemize}
		\item[\textup{(i)}] if there exists $T>0$ such that $v_{i,k} (T)$ is non-negative and non-zero for $i\in\Ical$, then $v_{m,k}(t) >0$ for all $t > T$.
		\item[\textup{(ii)}] if there exists $T>0$ such that $v_{i,k} (T)$ is non-positive and non-zero for $i\in\Ical$, then $v_{M,k}(t) <0$ for all $t > T$.
	\end{itemize}
\end{lemma}
\begin{proof}
	(i)
	By \eqref{Unifbd Ineq 2}, if $0< v_{m,k}(t) < \Cabk$ on some interval $(t_1, t_2)$, then $v_{m,k}$ is strictly increasing on $(t_1, t_2)$. 
	Moreover, if $v_{m,k}(T) \geq \Cabk$, then by Lemma \ref{Lm Upper Bound}, $v_{m,k}(t) \geq \Cabk$ for all $t\geq T$.
	Therefore, for $v_{m,k}(T) >0$, it is clear that $v_{m,k}(t)>0$ for all $t\geq T$. 
	We now consider the case where $v_{m,k}(T) =0$. 
	Since $v_{i,k}(T)$ is non-zero with respect to $i$, $v_{M,k}(T) >0$ for all $i\in\Ical$.
	Thus, from \eqref{Main_eq(2)}, we have
	\begin{align*}
		v_{m,k}^\prime(T) = \DpL v_{m,k} (T) = \sI{j} \csw |v_{j,k}(T)|^{p-2} v_{j,k}(T)>0,
	\end{align*}
	which implies that there exists a short interval $(T, T')$ such that $v_{m,k}(t)$ is strictly increasing on $(T, T')$.
	Finally, since $v_{m,k}(T') >0$, using the same argument above, we obtain $v_{m,k} (t) >0$ for all $t>T$.
	
	(ii)
	Since $-x_i$ and $-v_i$ are also solutions to \MSys{}, 
	by the arguments in (i), we can obtain the desired result.
\end{proof}

We have now discussed the convergence of the velocities of the agents in System \MSys.
Next, we provide estimations of the velocities, which imply asymptotic flocking.

\begin{lemma}\label{Lm limit 1}
	Let $k\in\Kcal$ be fixed. 
	Then, we have $\lim_{t\to\infty} v_{M,k}(t) \leq \Cabk$ and $\lim_{t\to\infty} v_{m,k}(t) \geq -\Cabk$ for any initial configuration.
\end{lemma}
\begin{proof}
	We first consider the simple case where $v_{M,k}(0) \leq \Cabk$.
	Then, by Lemma \ref{Lm Upper Bound}, it is clear that $\lim_{t\to\infty} v_{M,k}(t) \leq \Cabk$.
	We now assume that $v_{M,k}(0) > \Cabk$. 
	Then, there exists $T>0$ such that $v_{M,k}(t) > \Cabk$ for all $t\in[0,T)$.
	Moreover, it follows from \eqref{unifbd ineq} that $v_{M,k}$ is strictly decreasing on $(0, T)$.
	We consider the maximum $T$ that keeps $v_{M,k}$ strictly decreasing on the interval $(0, T)$.
	Then, we have the two cases of $T=\infty$ and $T<\infty$.
	In the case where $T<\infty$, since $T$ is the maximum time satisfying $v_{M,k}^\prime>0$, we have $v_{M,k}(T) = \Cabk$.
	Thus, by Lemma \ref{Lm Upper Bound}, we have the desired result. 
	For $T=\infty$, since $v_{M,k}$ is strictly decreasing and bounded below on $(0,\infty)$, $v_{M,k}$ converges to $\alpha\geq \Cabk$.
	We claim that $\alpha = \Cabk$. By contrast, suppose that $\alpha >\Cabk$.
	Then, it follows from \eqref{unifbd ineq} that
	\begin{align*}
		v_{M,k}^\prime(t)
		\leq 
		 - b_k \varphi_q(v_{M,k}(t)) \left(  \varphi_{r-q+1} (v_{M,k}(t)) - \frac{a_k}{b_k} \right)
		\leq 
		- b_k \alpha^{q-1} \left(  \alpha^{r-q} - \frac{a_k}{b_k} \right)<0,
		\quad t\geq 0.
	\end{align*}
	By integrating it from 0 to $t$, we deduce
	\begin{align*}
		v_{M,k}(t) - v_{M,k}(0) = \int_0^t v_{M,k}^\prime(s) ds \leq - b_k \alpha^{q-1} \left(  \alpha^{r-q} - \frac{a_k}{b_k} \right) t,
	\end{align*}
	for all $t\geq T$, but this leads to a contradiction as $t \to \infty$.
	Therefore, $\alpha = \Cabk$. 
	Hence, for any initial configuration, we have $\lim_{t\to\infty} v_{M,k}(t) \leq \Cabk$.
	Finally, since $-x_{ik}$ and $-v_{ik}$ are also solutions to system \MSys, it is clear that $\lim_{t\to\infty} v_{m,k}(t) \geq -\Cabk$.	
\end{proof}

\label{---}
\begin{theorem}\label{Th Consensus 1}
	For a fixed $k\in\Kcal$, 
	if $v_{i,k}(0)$ is non-negative and non-zero with respect to $i\in\Ical$,
	then $\lim_{t\to\infty} v_{i,k}(t) = \Cabk$ for all $i\in\Ical$.
\end{theorem}
\begin{proof}
	To prove this, it is sufficient to show by Lemma \ref{Lm limit 1} that $\lim_{t\to\infty} v_{m,k}(t) \geq \Cabk$.
	If $v_{m,k}(T) \geq \Cabk$ for some $T\geq 0$, then by Lemma \ref{Lm Upper Bound} and Lemma \ref{Lm limit 1}, we have
	\begin{align*}
		\Cabk \leq \lim_{t\to\infty} v_{m,k}(t) \leq \lim_{t\to\infty} v_{M,k}(t)\leq \Cabk.
	\end{align*}
	Thus, in this case, we have the desired result.
	We now consider the case where $v_{m,k}(t) < \Cabk$ for all $t\geq0$.
	Since $v_{m,k}(0) \geq 0$ and the initial configuration is not trivial, it follows from Lemma \ref{Lm Non-negativity} that $0< v_{m,k}(t)<\Cabk$ for all $t\geq0$, which implies $v_{m,k}$ is strictly increasing on $(0,\infty)$. Thus, $v_{m,k}$ converges to $\alpha \leq \Cabk$. 
	We claim that $\alpha = \Cabk$.
	To obtain a contradiction, suppose that $\alpha <\Cabk$.
	For any $\epsilon>0$, there exists $T>0$ such that 
	\begin{align*}
		\alpha-\epsilon < v_{m,k}(t) < \alpha,
	\end{align*}
	for $t\geq T$.
	By taking $\epsilon = \alpha/2$, we obtain
	\begin{align*}
		v_{m,k}^\prime(t) 
		\geq 
		b_k \varphi_q(v_{m,k}) \left( \frac{a_k}{b_k} - \left|\varphi_{r-q+1} (v_{m,k})\right| \right)
		\geq
		b_k \left( \frac{\alpha}{2} \right)^{q-1} \left( \frac{a_k}{b_k} - \alpha \right)>0,
		\quad t\geq T,
	\end{align*}
	which implies
	\begin{align*}
		v_{m,k} (t) - v_{m,k}(T) = \int_T^t v_{m,k}^\prime (s) ds
		\geq
		b_k \left( \frac{\alpha}{2} \right)^{q-1} \left( \frac{a_k}{b_k} - \alpha \right) (t-T),
		\quad t\geq T.
	\end{align*}	
	This leads to a contradiction with $\lim_{t\to\infty}v_{m,k}(t) = \alpha \leq \Cabk$.
	Hence, we conclude that $\lim_{t\to\infty} v_{m,k}(t) = \Cabk$.
\end{proof}

\begin{theorem}\label{Th Consensus 2}
	For a fixed $k\in\Kcal$,
	if $v_{i,k}(0)$ is non-positive and non-zero, then the system \MSys{} tends to a consensus and $\lim_{t\to\infty} v_{i,k} (t) = -\Cabk$ for all $i\in\Ical$.
\end{theorem}
\begin{proof}
	Since it is clear that $-x_i$ and $-v_i$ are also solutions to \MSys{}, it follows from Theorem \ref{Th Consensus 1} that $\lim_{t\to \infty} - v_i(t) = \Cabk$. 
	Hence, we have the desired result.	
\end{proof}

Note that in Theorem \ref{[Thm] Norm Convergence to Cab}, we obtained a result for the convergence of $\norm{v_i}{2}$, whereas in the above two theorems, we obtained a convergence result for each element of $v_i$. Therefore, by Theorems \ref{Th Consensus 1} and \ref{Th Consensus 2}, we can control the direction of $v_i$ by $a_k$ and $b_k$.

We can now show that system \MSys{} has flocking for the non-negative and non-zero initial velocity configuration. 
It is sufficint to show that $\max_{i,j\in\Ical}\norm{x_j - x_i}{2}$ is uniformly bounded with respect to $t$.
Hence, we first discuss estimates for $v_{m,k}$ and $v_{M,k}$ to prove it.

\begin{proposition}\label{[Lm] Estimate of v_Mk}
	For a fixed $k\in\Kcal$, if there exists $T\geq 0$ such that $v_{M,k}(t) > \Cabk$ for all $t\geq T$, then we have
	\begin{align*}
		\left| v_{M,k}(t) - \Cabk \right|
		\leq 
		\left| v_{M,k}(T) - \Cabk \right|
		\exp\left( 
		- b_k (r-q) \left(\frac{a_k}{b_k}\right)^{\frac{q-1}{r-q}} \xi_0 (t-T) 
		\right), 
		\quad t\geq T,
	\end{align*}
	where 
	\begin{align*}
  		\xi_0 := \left\{\begin{array}{ll}
  			\left(\frac{a_k}{b_k}\right)^{\frac{r-q-1}{r-q}},&\text{ if } r-q-1\geq 0,\\
  			\left(v_{M,k}(T) \right)^{r-q-1}, & \text{ if } r-q-1<0.
  		\end{array}\right.
	\end{align*}
\end{proposition}
\begin{proof}
	We first denote $\Cabk$ briefly by $C$. 
	Since we assume that $v_{M,k}(t) > C$ for all $t\geq T$, we have
	\begin{align*}
		\frac{1}{2} \frac{d}{dt}(v_{M,k}(t) - C)^2 
		=& v_{M,k}^\prime(t) (v_{M,k}(t) - C)\\
		=& \MDdpl v_{M,k} (t) (v_{M,k}(t) - C) + \left(a_k v_{M,k}^{q-1}(t) - b_k v_{M,k}^{r-1}(t) \right) (v_{M,k}(t) - C).
	\end{align*}
	Since $\MDdpl v_{M,k} (t) (v_{M,k}(t) - C)\leq 0$, $t\geq T$, we obtain
	\begin{align*}
		\frac{1}{2} \frac{d}{dt}(v_{M,k}(t) - C)^2 
		\leq &
		- b_k v_{M,k}^{q-1}(t) \left( v_{M,k}^{r-q}(t) - C^{r-q} \right)(v_{M,k}(t) - C)<0,
		\quad t \geq T.
	\end{align*}
	Moreover, since $v_{M,k}(t)$ is strictly decreasing on $(T,\infty)$ by \eqref{unifbd ineq}, we also obtain $C < v_{M,k}(t) <v_{M,k}(T)$ for all $t\geq T$.
	For each $t\geq T$, by applying the mean value theorem to the term $v_{M,k}^{r-q}(t) - C^{r-q}$, there exists $\xi(t) \in (C,v_{M,k}(t))\subset (C,v_{M,k}(T))$ such that
	\begin{align*}
		\frac{v_{M,k}^{r-q}(t) - C^{r-q}}{v_{M,k}(t) - C} = (r-q) \left( \xi(t) \right)^{r-q-1}
		>(r-q) \xi_0.
	\end{align*}
	Thus, the following ordinary differential inequality is obtained:
	\begin{align}
		\frac{1}{2} \frac{d}{dt}(v_{M,k}(t) - C)^2 
		\leq &
		- b_k  (r-q) \xi_0 v_{M,k}^{q-1}(t)    (v_{M,k}(t) - C)^2 \nonumber \\
		< &
		- b_k  (r-q) \xi_0 C^{q-1}   (v_{M,k}(t) - C)^2, \label{prop estimate 1 pf 010}
		\quad t\geq T.
	\end{align}	
	Finally, by solving \eqref{prop estimate 1 pf 010}, we have
	$\left(v_{M,k}(t) - C\right)^2 
	\leq 
	\left(v_{M,k}(T) - C\right)^2 
	\exp\left( 
	-2 b_k (r-q) C^{q-1} \xi_0 (t-T) 
	\right) 
	$, $t\geq T$, which completes the proof.
\end{proof}

\begin{proposition}\label{[Lm] Estimate of v_mk}
	For a fixed $k\in\Kcal$, if there exists $T\geq 0$ such that $0 < v_{m,k}(t) < \Cabk$ for all $t\geq T$, then we have
	\begin{align*}
		\left| v_{m,k} (t) - \Cabk \right| 
		\leq 
		\left| v_{m,k} (T) - \Cabk \right| 
		\exp \Bigg(- b_k (r-q)  (v_{m,k}(T))^{q-1} \xi_1 (t- T) \Bigg), 
		\quad t\geq T,
	\end{align*}
	where 
	\begin{align*}
  		\xi_1 := \left\{ \begin{array}{ll}
  			\left(v_{m,k}(T) \right)^{r-q-1},&\text{ if }r-q-1\geq 0,\\
  			\left(\frac{a_k}{b_k}\right)^{\frac{r-q-1}{r-q}},&\text{ if }r-q-1<0.
  		\end{array}\right.
	\end{align*}
\end{proposition}
\begin{proof}
	Since $ 0< v_{m,k}(t) < C := \Cabk$, $t\geq T$,
	$v_{m,k}$ is strictly increasing on $(T,\infty)$ from \eqref{Unifbd Ineq 2}.
	Applying the same method as in the proof of Theorem \ref{[Lm] Estimate of v_Mk},
	we have
	\begin{align*}
		\frac{1}{2} \frac{d}{dt}(v_{m,k}(t) - C)^2
		< &
		- b_k (r-q) \xi_1 (v_{m,k}(T))^{q-1}  (v_{m,k}(t) - C)^2,
		\quad t \geq T,
	\end{align*}	
	which implies
	\begin{align*}
		(v_{m,k}(t) - C)^2 
		\leq
		(v_{m,k}(T) - C)^2
		\exp\left(
			- 2 b_k (r-q)(v_{m,k}(T))^{q-1}  \xi_1  (t-T)
		\right),
	\end{align*}
	for all $t\geq T$.
\end{proof}

\begin{theorem}\label{Th Uniform Boundedness of X}
	For a fixed $k\in\Kcal$, if $v_{i,k}(0)$ is non-negative and non-zero, then the system \MSys{} has asymptotic flocking.
\end{theorem}
\begin{proof}
	We perform simple calculations to arrive at the identity:
	\begin{align*}
		\frac{d}{dt} \left( \sumI{i,j} (x_{ik} - x_{jk})^2 \right)^{\frac{1}{2}}
		\leq &
		\left( \sumI{i,j} (v_{ik} - v_{jk})^2 \right)^{\frac{1}{2}}.
	\end{align*}	
	By integrating from $t_0$ to $t$, we have
	\begin{align}\label{Th UBdd of X Pf 010}
		\left( \sumI{i,j} (x_{ik}(t) - x_{jk}(t))^2 \right)^{\frac{1}{2}}
		\leq
		\left( \sumI{i,j} (x_i(t_0) - x_j(t_0))^2 \right)^{\frac{1}{2}}
		+ \int_{t_0}^{t}
		\left( \sumI{i,j} (v_{ik}(s) - v_{jk}(s))^2 \right)^{\frac{1}{2}} ds.
	\end{align}
	By Proposition \ref{Prop Trichotomy}, there exists $T\geq 0$ that allows only one of the following:
	\begin{align*}
	\textup{(i) } v_{M,k}(t) \leq \Cabk,~ t\geq T, 
	\quad
	\textup{(ii) } v_{m,k}(t) \geq \Cabk,~ t\geq T,
	\quad
	\textup{(iii) } v_{m,k}(t) < \Cabk < v_{M,k}(t),~ t\geq T.
	\end{align*}
	We first consider Case (i).
	Since $v_{i,k}(0)$ is non-negative and non-zero, by Lemma \ref{Lm Non-negativity}, $v_{m,k}(t) >0$ for all $t > 0$.
	Thus, we have $0<v_{m,k}(t)\leq v_{M,k}(t)\leq \Cabk(:=C)$ for $t > T$, which implies 
	\begin{align*}
		(v_{i,k}(t) - C)^2 \leq (v_{m,k}(t) -C)^2, \quad t > T,
	\end{align*}
	for all $i\in \Ical$.
	Therefore, by Proposition \ref{[Lm] Estimate of v_mk}, we have
	\begin{align*}
		\left( \sumI{i,j} (v_{ik}(t) - v_{jk}(t))^2 \right)^{\frac{1}{2}}
		\leq &
		\left(
		\sumI{i,j} 2 \left( (v_{ik}(t) - C)^2 + (C - v_{jk}(t))^2 \right)
		\right)^{\frac{1}{2}}\\
		& \leq
		\left(
		4 N^2 (v_{m,k}(t) -C)^2
		\right)^{\frac{1}{2}}\\
		&\leq
		2N C_{1,T}
		\exp\left(- C_{2,T} (t- T)\right),
	\end{align*}
	where $C_{1,T} = \left| v_{m,k} (T) - C \right|$, and $C_{2,T}= b_k (r-q)  (v_{m,k}(T))^{q-1} \xi_1$.
	By replacing $t_0$ with $T$ in \eqref{Th UBdd of X Pf 010}, we obtain
	\begin{align*}
		\int_{T}^{t}
		\left( \sumI{i,j} (v_{ik}(s) - v_{jk}(s))^2 \right)^{\frac{1}{2}} ds
		<&
		\frac{\left( 2N^2 C_{1,T} \right)^{\frac{1}{2}}}{C_{2,T}}, 
		\quad t \geq T.
	\end{align*}
	Hence, $\sumI{i,j}(x_{ik}(t) - x_{jk}(t))^2 $ is uniformly bounded.
	Therefore, the system \MSys{} has asymptotic flocking.
	Similarly, Case (ii) can be proved by Proposition \ref{[Lm] Estimate of v_Mk}.
	Finally, for (iii), since $v_{m,k}(t) >0$ for all $t>0$, by Propositions \ref{[Lm] Estimate of v_Mk} and \ref{[Lm] Estimate of v_mk}, we have the desired result.
\end{proof}

Next, we add the assumption that $p<2$ and discuss initial configurations to induce flocking in finite time.
We also provide initial distribution conditions for $v_{i,k}$ to converge to $\Cabk$.
As previously mentioned, if there exists a finite flocking time, flocking occurs; hence, we focus on the existence of a finite flocking time.

\label{---}
\begin{theorem}\label{Thm Finite Extinction Time}
	For $1<p<2$, $p<q<r$, and a fixed $k\in\Kcal$, let $a_k$ be sufficiently small.
	If there exists $t_0 >0$ ($t_0$ may be infinite) such that
	\begin{align}\label{Thm FET Eq 010}
		(v_{M,k} (0) - v_{m,k} (0))^{2-p}
  		<
  		\left\{ \begin{array}{ll}
  			2(2-p)\int_0^{t_0} \min_{i,j\in\Ical} \CSweightT dt , & \text{if } q \geq 2,\\
  			(2-p) \int_0^{t_0} 2 
  			\left(
  			\min_{i,j\in\Ical} \CSweightT - a_k C_1 B_k^{q-p} 
  			\right) 
  			dt & \mbox{if } 1<q<2,
  		\end{array}\right.
	\end{align}
	then there exists $T\geq0$ such that $v_{M,k} (t) - v_{m,k} (t) = 0$ for all $t\geq T$, where $B_k:=B_{m,k} + B_{M,k}$. 
	Here the definitions of $B_{m,k}$ and $B_{M,k}$ is in Proposition \ref{Prop Uniformly Bounded for v}.
\end{theorem}
\begin{proof}
	We first calculate
	\begin{align}
  		\frac{d}{dt}(v_{M,k} - v_{m,k}) 
  		=& 
  		\DpL v_{M,k} - \DpL v_{m,k} + a_k ( \varphi_q(v_M)  - \varphi_q(v_m)) 
  		- b_k( \varphi_r(v_M)  - \varphi_r(v_m)) \nonumber \\
  		\leq &
  		-2 \pmin ( v_{M,k} - v_{m,k} )^{p-1}
  		+ a_k C_1 (v_{M,k} - v_{m,k})^{1-\delta} ( |v_{M,k}| + |v_{m,k}| )^{q-2 + \delta}. \label{Thm FET Pf 010}
	\end{align}
	The proof falls naturally into two cases: $q\geq2$ and $1<q<2$.
	
	We first consider the case where $q\geq 2$.
	By Proposition \ref{Prop Uniformly Bounded for v}, we have $|v_{M,k}(t)| + |v_{m,k}(t)| \leq B_{M,k} + B_{m,k} = B_k$ for all $t\geq 0$; therefore, we have
	\begin{align*}
  		\frac{d}{dt}(v_{M,k} - v_{m,k}) 
  		\leq & 
  		-2 \pmin ( v_{M,k} - v_{m,k} )^{p-1}
  		+ a_k C_1 B_k^{q-2} (v_{M,k} - v_{m,k}).
	\end{align*}
	Here, we take $\delta =0$.
	Then, $v_{i,k}$ satisfies 
	\begin{align*}
  		(v_{M,k}(t) - v_{m,k}(t))^{2-p} e^{- a_k C_1 B_k^{q-2} t} 
  		-(v_{M,k}(0) - v_{m,k}(0))^{2-p}
  		\leq 
  		-2 (2-p) \int_0^t \pmin (s) e^{- a_k C_1 B_k^{q-2} s}  ds.
	\end{align*}
	for all $t\geq 0$.
	Since $\int_0^{t_0} \pmin (s) e^{- a_k C_1 B_k^{q-2} s}  ds \to \int_0^{t_0} \pmin (s)  ds$ as $a_k \to 0$, we can take a small enough $a_k$ such that
	\begin{align*}
  		(v_{M,k}(0) - v_{m,k}(0))^{2-p} 
  		< 
  		\int_0^{t_0} \pmin (s) e^{- a_k C_1 B_k^{q-2} s}  ds.
	\end{align*}
	Thus, there exists $T>0$ such that 
	\begin{align*}
  		(v_{M,k}(0) - v_{m,k}(0))^{2-p} = 2 (2-p) \int_0^T \pmin (s) e^{- a_k C_1 B_k^{q-2} s}  ds,
	\end{align*}
	which implies $v_{M,k} (t) - v_{m,k} (t) = 0$ for all $t\geq T$.
	
	When $1<q<2$, by taking $\delta = 2-q$ in \eqref{Thm FET Pf 010}, we have
	\begin{align*}
  		\frac{d}{dt}(v_{M,k} - v_{m,k}) 
  		\leq &
  		-2 \pmin ( v_{M,k} - v_{m,k} )^{p-1}
  		+ a_k C_1 (v_{M,k} - v_{m,k})^{q-1}\\
  		=&
  		-2 \pmin ( v_{M,k} - v_{m,k} )^{p-1}
  		+ a_k C_1 (v_{M,k} - v_{m,k})^{p-1} (v_{M,k} - v_{m,k})^{q-p}.
	\end{align*}
	Since $p<q$, $(v_{M,k} - v_{m,k})^{q-p}$ is uniformly bounded by $B_k^{p-q}$. 
	Therefore, we obtain 
	\begin{align*}
  		\frac{d}{dt}(v_{M,k} - v_{m,k}) 
  		\leq &
  		- \left( 2 \pmin - a_k C_1 B_k^{q-p} \right) (v_{M,k} - v_{m,k})^{p-1},
	\end{align*}
	which implies
	\begin{align*}
		(v_{M,k} (t) - v_{m,k} (t) )^{2-p} 
		\leq &
		(v_{M,k} (0) - v_{m,k} (0) )^{2-p} 
		- (2-p) \int_0^t 2 \pmin (s) - a_k C_1 B_k^{q-p} ds,
		\quad t\geq 0.
	\end{align*}
	Thus, we have the desired result by \eqref{Thm FET Eq 010}.	
\end{proof}

We note that by Lemma \ref{Lm Non-negativity}, we can obtain the trichotomy result:
for an arbitrary initial configuration,
\begin{itemize}
	\item[(i)] there exists $T\geq 0$ such that $v_{i,k}(t)>0$ for all $i\in\Ical$ and $t\geq T$,
	\item[(ii)] there exists $T\geq 0$ such that $v_{i,k}(t)<0$ for all $i\in\Ical$ and $t\geq T$,
	\item[(iii)] $v_{m,k}(t) \leq 0 \leq v_{M,k}(t)$ for all $t\geq 0$.
\end{itemize} 
Moreover, by Theorems \ref{Th Consensus 1} and \ref{Th Consensus 2}, for cases (i) and (ii), we obtain $\lim_{t\to\infty} v_{i,k} (t) = \Cabk$ and $\lim_{t\to\infty} v_{i,k} (t) = -\Cabk$, respectively.
For case (iii), if we assume all the conditions in Theorem \ref{Thm Finite Extinction Time}, then Theorem \ref{Thm Finite Extinction Time} implies $\lim_{t\to\infty} v_{i,k} (t) = 0$.
Thus, the convergence value of $v_{i,k}$ is determined by the given initial configuration to be one of the three values $-\Cabk$, $0$, $\Cabk$ above.
In particular, we are interested in the initial configuration that allows a convergence to $\Cabk$.
By Theorem \ref{Th Consensus 1}, for a fixed $k\in\Kcal$, if $v_{i,k}(0)$ is non-negative and non-zero, then $\lim_{t\to\infty} v_{i,k} = \Cabk$ for all $i$. Hence, in the next result, we focus on the case where $v_{m,k}(0) < v_{M,k}(0)$.

\label{---}
\begin{proposition}\label{[Prop] Convergence to Cabk}\label{Lm Estimations of v_M and v_m}
	For $1<p<2$, $p<q<r$, and a fixed $k\in\Kcal$ satisfying $v_{m,k}(0)<0<v_{M,k}(0)$, 
	let $a_k$ be small enough and there exist $t_0>0$ ($t_0$ may be infinite) such that
	\begin{align}\label{Prop Conv Cabk 010}
  		|v_{m,k}(0)|^{2-p}
  		<
  		\left\{ \begin{array}{ll}
  			(2-p)\int_0^{t_0} \min_{i,j\in\Ical} \CSweightT dt , & \text{if } q \geq 2,\\
  			(2-p) \int_0^{t_0} \left( \min_{i,j\in\Ical} \CSweightT - a_k B_{m,k}^{q-p} \right) dt & \mbox{if } 1<q<2.
  		\end{array}\right.
	\end{align}
	If $|v_{m,k}(0)|$ is small enough, then $\lim_{t\to\infty} v_{i,k} (t) = \Cabk$ for all $i\in\Ical$.
\end{proposition}
\begin{proof}
	For a given $k$, 
	since $v_{M,k}(0) >0$, there exists $T>0$ (maybe infinity) such that $v_{M,k}(t) >0$ for all $t\in[0, T)$.
	We now suppose that $v_{m,k}(t) < 0$ for all $t\in[0, T)$.
	Then, we have 
	\begin{align*}
  		( - v_{m,k})^\prime 
  		\leq &
  		- \pmin ( v_{M,k} - v_{m,k} )^{p-1} + a_k \varphi_q(-v_{m,k}) - b_k \varphi_r(-v_{m,k})\\
  		\leq &
  		- \pmin ( - v_{m,k} )^{p-1} + a_k ( - v_{m,k})^{q-1}
	\end{align*}
	on $[0,T)$, implying
	\begin{itemize}
		\item if $q\geq 2$, then 
			\begin{align}\label{Prop Conv Cabk Pf 010}
  				(-v_{m,k}(t))^{2-p}
  				\leq
  				(-v_{m,k}(0))^{2-p} - (2-p) \int_0^t \pmin (s) e^{-a_k B_{m,k}^{q-2} s}ds,
			\end{align}
		\item if $1<q<2$, then 
			\begin{align}\label{Prop Conv Cabk Pf 020}
  				(-v_{m,k}(t))^{2-p}
  				\leq
  				(-v_{m,k}(0))^{2-p} - (2-p) \int_0^t \left( \pmin (s) - a_k B_{m,k}^{q-p}\right) ds,
			\end{align}
	\end{itemize}
	for all $t\in[0,T)$.
	By applying a similar argument to the proof of Theorem \ref{Thm Finite Extinction Time}, 
	there exists $T^\prime >0$ such that the right-hand side of \eqref{Prop Conv Cabk Pf 010} or \eqref{Prop Conv Cabk Pf 020} is zero at $t=T'$ for each case.
	We note that $T'\to 0$ as $|v_{m,k}(0)|\to 0$.
	Therefore, we can choose a sufficiently small $|v_{m,k}(0)|$ such that $T'<T$, which is a contraction. 
	Thus, there exists $T_0<T$ such that $v_{M,k}(T_0) >0$ and $v_{m,k}(T_0)\geq 0$ for all $k\in\Kcal$. Thus, by Theorem \ref{Th Consensus 1}, $v_{i,k} (t) \to \Cabk$ for all $i\in\Ical$.
\end{proof}

\section{Applications}

In this section, we show that the regular communication weight $\psi_R$ satisfies \eqref{Thm: Eq 010}, \eqref{Thm FET Eq 010}, and \eqref{Prop Conv Cabk 010}.
We first present an estimate for the regular communication weight $\psi_R$, which has been proved in \cite[Lemma 4.3]{kim2020complete}.

\begin{lemma}\label{Lm Regular Communication Weight}\label{Lm RCW}
	For $p>1$ and $\beta>0$, the regular communication weight $\psi_{R}$ satisfies
	\begin{align*}
  		\Rcwt =\left( \frac{1}{1 + \Lnorm{x_j (t) - x_i (t)}^2}\right)^{\frac{\beta}{2}}
  		> 
  		\left[ \psi_R^{- \frac{1}{\beta}}\left( \norm{x_j (0) - x_i (0)}{2} \right) + 
  		\int_0^t \norm{v_j (s) - v_i(s)}{2}ds
  		\right]^{-\beta}
	\end{align*}
	for all $i,j\in\Ical$ and $t\geq0$.
\end{lemma}
\begin{proof}
	The proof of this lemma is obtained by direct calculations.
	Hence, we only provide a sketch of the proof. 
	For more details, see \cite{kim2020complete}.
	We first calculate
	\begin{align*}
  		\frac{d}{dt} \Rcwt
  		\geq
  		-\beta \left( \Rcwt \right)^{1+\frac{1}{\beta}}\norm{v_j(t) - v_i(t)}{2} ,
	\end{align*}
	which implies
	\begin{align*}
  		\Rcwt \geq 
  		\left[ \psi_R^{- \frac{1}{\beta}}\left( \norm{x_j (0) - x_i (0)}{2} \right) + 
  		\int_0^t \norm{v_j (s) - v_i(s)}{2}ds
  		\right]^{-\beta}
	\end{align*}
	for all $i,j\in\Ical$.
\end{proof}

We consider two cases, $0<\beta\leq 1$ and $\beta>1$.

\begin{proposition}\label{[Prop] Regular Communication Weight 1}
	For $1<p<2$, $p<q<r$, and $k\in\Kcal$, suppose that $a_k$ is sufficiently small.
	If $0< \beta \leq 1$, then 
	the regular communication weight $\psi_R$ satisfies the equalities \eqref{Thm: Eq 010}, \eqref{Thm FET Eq 010}, and \eqref{Prop Conv Cabk 010}.
\end{proposition}
\begin{proof}
	Since $|v_{i,k}| \leq B_k$ for $i\in\Ical$ and $k\in\Kcal$, we have 
	\begin{align*}
  		\norm{v_j(t) - v_i(t)}{2} \leq 2 \sumK{k} B_k =: M,
	\end{align*}
	for all $i,j\in\Ical$.
	Therefore, it follows from Lemma \ref{Lm RCW} that 
	\begin{align*}
  		\Rcwt 
  		> 
  		\left[ \psi_R^{- \frac{1}{\beta}}\left( \norm{x_j (0) - x_i (0)}{2} \right) + 
  		M t 
  		\right]^{-\beta}
	\end{align*}
	for all $i,j\in\Ical$ and $t\geq0$.
	Thus, we have
	\begin{align}\label{Prop RCW 1 Pf 010}
  		\min_{i,j\in\Ical} \Rcwt
  		>
  		\min_{i,j\in\Ical}
  		\left[ \psi_R^{- \frac{1}{\beta}}\left( \norm{x_j (0) - x_i (0)}{2} \right) + 
  		M t 
  		\right]^{-\beta}
  		=
  		\left[ \left( \min_{i,j\in\Ical} \psi_R\left( \norm{x_j (0) - x_i (0)}{2} \right) \right)^{- \frac{1}{\beta}} + 
  		M t 
  		\right]^{-\beta}
	\end{align}
	for all $t\geq0$.
	For simplicity, we write $\pmin (t)$ instead of $\min_{i,j\in\Ical} \Rcwt$.
	Let us consider the case where $0<\beta <1$.
	Then, if $q\geq 2$, we have
	\begin{align*}
  		\int_0^t \pmin (s) ds 
  		> \frac{1}{M(1-\beta)} 
  			\left( \pmin^{- \frac{1}{\beta}}(0) + Mt \right)^{1-\beta} 
  		- \frac{1}{M(1-\beta)} 
  			\left( \pmin^{- \frac{1}{\beta}}(0)\right)^{1-\beta} 
  		\to \infty,
  		\text{ as } t \to \infty.
	\end{align*}
	Hence, $\psi_R$ satisfies \eqref{Thm: Eq 010}, \eqref{Thm FET Eq 010}, and \eqref{Prop Conv Cabk 010}.
	In the case where $1<q<2$,
	we have
	\begin{align*}
  		(2-p) \int_0^t 2 \pmin (s) - a_k C_1 B_k^{q-p} ds 
  		> &
  		\frac{2(2-p)}{M(1-\beta)} 
  			\left( \pmin^{- \frac{1}{\beta}}(0) + Mt \right)^{1-\beta} 
  		- \frac{2(2-p)}{M(1-\beta)} 
  			\left( \pmin^{- \frac{1}{\beta}}(0)\right)^{1-\beta} 
  		- (2-p) a_k C_1 B_k^{q-p} t\\
  		> &
  		\frac{2(2-p)}{M(1-\beta)} 
  			\left( Mt \right)^{1-\beta}
  			- \frac{2(2-p)}{M(1-\beta)} 
  			\left( \pmin^{- \frac{1}{\beta}}(0)\right)^{1-\beta} 
  			- (2-p) a_k C_1 B_k^{q-p} t.
	\end{align*}
	Taking 
	\begin{align*}
  		g(t) 
  		=& 
  		\frac{2(2-p)}{M(1-\beta)} 
  			\left( Mt \right)^{1-\beta}
  			- (2-p) a_k C_1 B_k^{q-p} t,
	\end{align*}
	then it is easy to see that $g$ is maximized at $t=t_0 := \left(  \frac{2(2-p)M^{1-\beta}}{M(2-p)a_k C_1 B_k^{q-p}} \right)^{\frac{1}{\beta}}$.
	Since $B_k$ depends on $a_k$, we deduce $t_0\to \infty$ as $a_k \to 0$, which implies
	\begin{align*}
  		g(t_0) 
  		=& 
  		\left(
  		\frac{2(2-p)}{M(1-\beta)} 
  			M^{1-\beta} 
  			- (2-p)a_k C_1 B_k^{q-p} t_0^{\beta}
  		\right) t_0^{1-\beta}\\
  		=&
  		\left(  \frac{2(2-p)M^{1-\beta}}{M} \right)
  		\left(
  		\frac{1}{1-\beta} 
  			- 1
  		\right) t_0^{1-\beta} \to \infty \text{ as } a_k \to 0.
	\end{align*}
	Therefore, we can choose a sufficiently small $a_k$ such that
	\begin{align*}
  		(v_{M,k} (0) - v_{m,k} (0))^{2-p}
  		< &
  		g(t_0) - \frac{2(2-p)}{M(1-\beta)} 
  			\left( \pmin^{- \frac{1}{\beta}}(0)\right)^{1-\beta}\\
  		< &
  		(2-p) \int_0^{t_0} \left( 2 \pmin (s) - a_k C_1 B_k^{q-p} \right) ds.
	\end{align*}
	Hence, $\psi_R$ satisfies \eqref{Thm FET Eq 010} and \eqref{Prop Conv Cabk 010}.
	
	In the case where $\beta=1$, 
	we can obtain the desired result using the arguments, except that only the integral $\int_0^t \pmin$ is
	\begin{align*}
  		\int_0^t \pmin (s) ds 
  		= 
  		\ln \left( \pmin^{-1} (0) +Mt \right) - \ln \left( \pmin^{-1} (0)\right).
	\end{align*}
	Hence, we omit the proof of the case where $\beta=1$.
\end{proof}

\begin{proposition}\label{[Prop] Regular Communication Weight 2}
	For $1<p<2$, $q<p<r$, and $k\in\Kcal$, suppose that $a_k$ is sufficiently small and
	\begin{align}\label{Prop RCW 2 H 010}
  		\frac{2(2-p)}{M(\beta-1)} 
  		\left( \Rcwtzero \right)^{\frac{\beta-1}{\beta}}
  		>
  		(v_{M,k} (0) - v_{m,k} (0))^{2-p}.
	\end{align}
	If $\beta >1$, then the regular communication weight $\psi_R$ satisfies the inequality \eqref{Thm FET Eq 010}.
\end{proposition}
\begin{proof}
	We first consider the case where $q\geq 2$.
	It follows from \eqref{Prop RCW 1 Pf 010} that
	\begin{align}\label{Prop RCW 2 Pf 010}
  		\int_0^t \pmin (s) ds 
  		>  - \frac{1}{M(\beta - 1)} 
  			\left( 
  			\frac{1}{\pmin^{- \frac{1}{\beta}}(0) + Mt }
  			\right)^{\beta-1} 
  		+ \frac{1}{M(\beta-1)} 
  			\pmin^{\frac{\beta-1}{\beta}}(0).
	\end{align}
 	The right hand side of \eqref{Prop RCW 2 Pf 010} converges to $\frac{1}{M(\beta-1)} \pmin^{\frac{\beta-1}{\beta}}(0)$ as $t \to \infty$.
	Thus, by \eqref{Prop RCW 2 H 010}, we have
	\begin{align*}
  		2(2-p)\int_0^\infty \pmin (s) ds 
  		\geq 
  		\frac{2(2-p)}{M(\beta-1)} \pmin^{\frac{\beta-1}{\beta}}(0)
  		>
  		(v_{M,k} (0) - v_{m,k} (0))^{2-p}.
	\end{align*}

	In the case where $1<q<2$, we have
	\begin{align*}
  		(2-p) \int_0^{t} \left(2 \pmin (s) - a_k C_1 B_k^{q-p} \right) ds
  		> &
  		- \frac{2(2-p)}{M(\beta - 1)} 
  			\left( 
  			\frac{1}{\pmin^{- \frac{1}{\beta}}(0) + Mt }
  			\right)^{\beta-1} 
  		+ \frac{2(2-p)}{M(\beta-1)} 
  			\pmin^{\frac{\beta-1}{\beta}}(0)
  		- (2-p) a_k C_1 B_k^{q-p} t\\
  		>&
  		- \frac{2(2-p)}{M(\beta - 1)} 
  			\left( 
  			\frac{1}{ Mt }
  			\right)^{\beta-1}
  		+ \frac{2(2-p)}{M(\beta-1)} 
  			\pmin^{\frac{\beta-1}{\beta}}(0)
  		- (2-p) a_k C_1 B_k^{q-p} t.
	\end{align*}
	Taking 
	\begin{align*}
  		g(t) = - \frac{2(2-p)}{M(\beta - 1)} 
  			\left( 
  			\frac{1}{ Mt }
  			\right)^{\beta-1}
  			- (2-p) a_k C_1 B_k^{q-p} t,
	\end{align*}
	then the maximizer $t_0$ of $g$ is 
	\begin{align*}
  		t_0
  		=
  			\left(
  			\frac{2(2-p)}{ (2-p) a_k C_1 B_k^{q-p} M^\beta } 
  			\right)^{\frac{1}{\beta}}.
	\end{align*}
	Then, we have 
	\begin{align*}
  		g(t_0) 
  		=
  		-
  		\frac{2(2-p)}{M^\beta} 
  		\left(
  		 \frac{1}{\beta - 1} 
  			+  1
  		\right) 
  		\left(
  			\frac{ (2-p) a_k C_1 B_k^{q-p} M^\beta } {2(2-p)}
  		\right)^{\frac{\beta -1}{\beta}}
  		\longrightarrow 0 \text{ as } a_k \to 0.
	\end{align*}
	Therefore, by \eqref{Prop RCW 2 H 010}, there exists a sufficiently small $a_k$ such that
	\begin{align*}
  		(2-p) \int_0^{t_0} \left( 2 \pmin (s) - a_k C_1 B_k^{q-p} \right) ds
  		> & 
  		g(t_0)
  		+ \frac{2(2-p)}{M(\beta-1)} 
  			\pmin^{\frac{\beta-1}{\beta}}(0) \\
  		> &
  		(v_{M,k} (0) - v_{m,k} (0))^{2-p}.
	\end{align*}
	Thus, we have the desired result.
\end{proof}

\begin{remark}
	We note that if we change \eqref{Prop RCW 2 H 010} into 
	\begin{align*}
  		\frac{2C_m (2-p)}{M(\beta-1)} \pmin^{\frac{\beta-1}{\beta}}(0)
  		>
  		\norm{v_M (0) - v_m (0)}{2} ^{2-p},
  		\quad \text{or} \quad
  		\frac{2-p}{M(\beta-1)} \pmin^{\frac{\beta-1}{\beta}}(0)
  		>
  		|v_{m,k} (0)|^{2-p},
	\end{align*}
	Then, by the argument of the proof in Proposition \ref{[Prop] Regular Communication Weight 2}, we can show that the regular communication weight $\psi_R$ satisfies \eqref{Thm: Eq 010}, or \eqref{Prop Conv Cabk 010}.
\end{remark}


\label{====================================}
\section{Numerical Simulations}\label{Section: Numer Simul}

Several numerical simulations were used to confirm the analytical results of this study; the simulations were based on the Runge-Kutta method to illustrate the main results of the study.
All simulations were run using MATLAB’s ODE solver ‘‘ode89,'' and we always considered 
\begin{align*}
	d= 2   , \quad
	p= 1.5 , \quad
  	q= 2.5 , \quad
  	r= 3.5.  
\end{align*}
In addition, we considered the regular communication weight $\psi_R$ with $\beta = 0.5$ as the function $\psi$ to minimize the initial configuration conditions.
If we consider $\beta >1$, then we must check whether the initial condition \eqref{Prop RCW 2 H 010} is true.

\begin{example}\label{Ex1}
	In this example, we run a simulation of the results of Theorem \ref{(Thm) Consensus for Rayleigh friction} and Proposition \ref{[Prop] Positive Convergence}. 
	We first consider $20$ agents (i.e., $N=20$). The initial configuration is given as Figure \ref{fig1} (a), and the position $x_i$ and velocity $v_i$ are created with uniformly distributed random numbers ranging from $-10$ to $10$.
	We also take $a = 0.1$ and $b=0.05$.
	
	Figure \ref{fig1} (b) is a graph of the trajectories of all agents, and Figure \ref{fig1} (c) shows that $\norm{v_i}{2}$ converges to $\Cab = 2$ around $t=1$ in finite time. 
	
\begin{figure}[thb!]
\centering
	\includegraphics[scale=0.35]{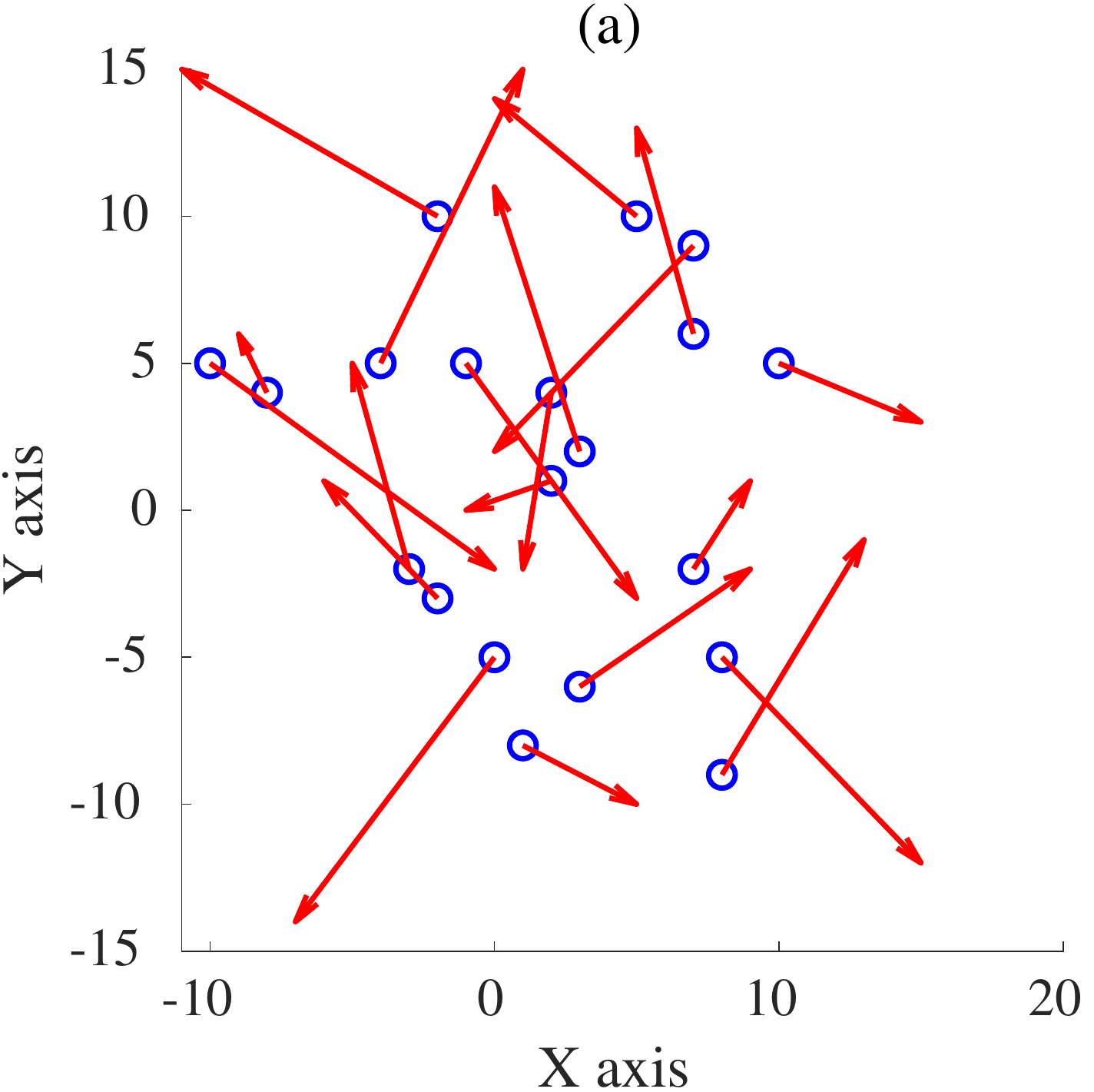}
	\includegraphics[scale=0.35]{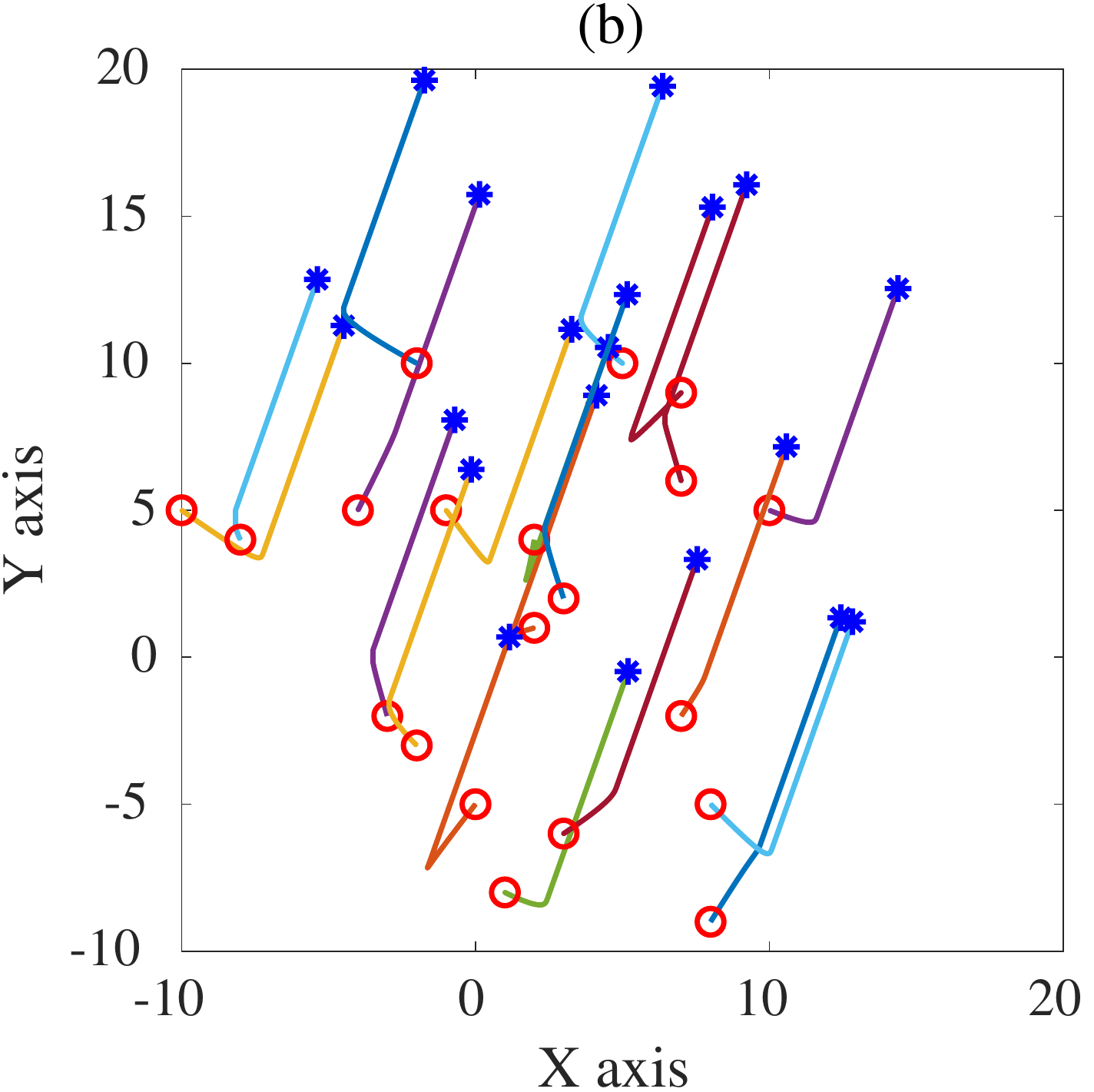}
	\includegraphics[scale=0.35]{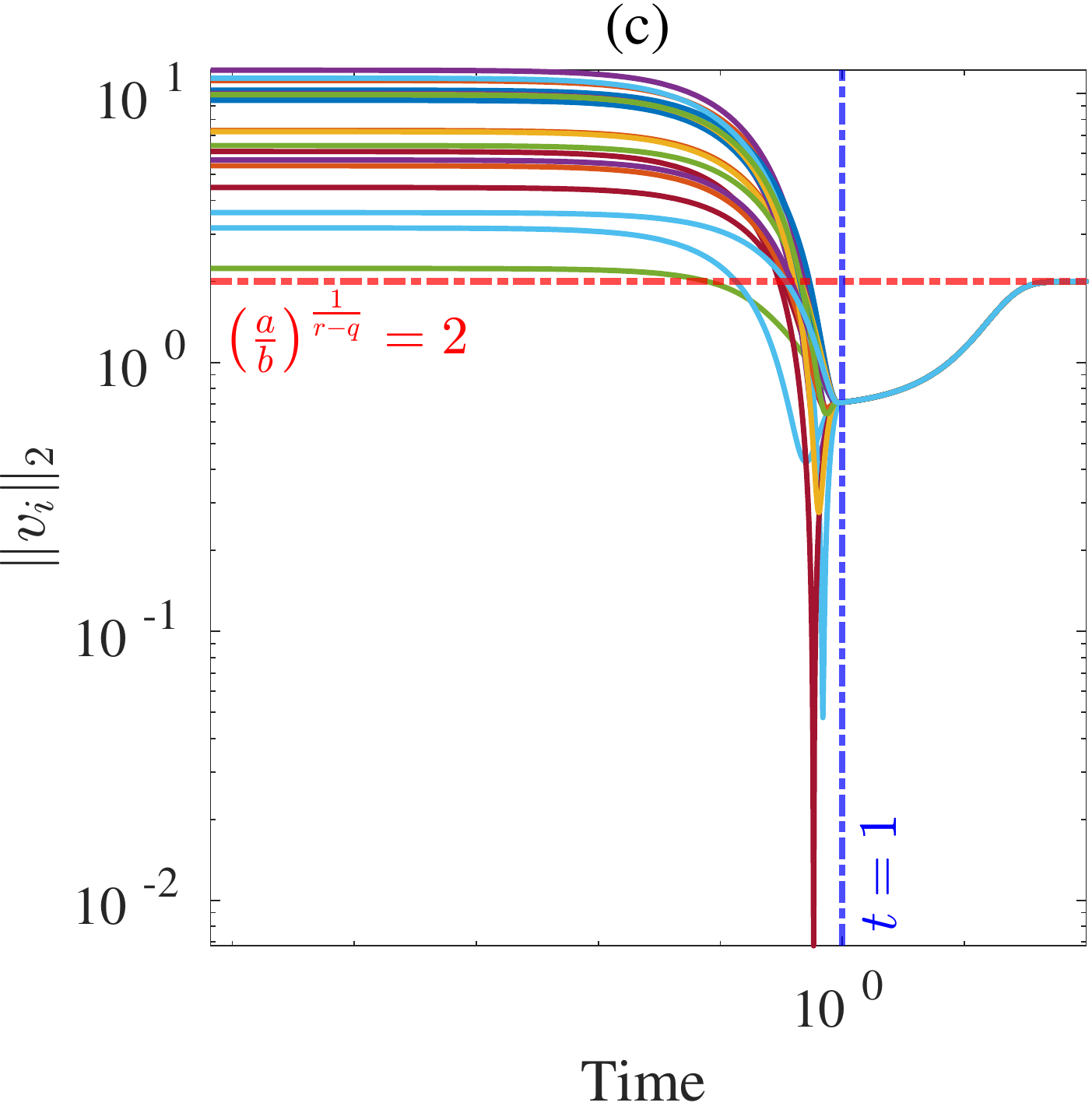}
	\caption{
	(a) is the initial configuration, and the blue circles and red arrows are the agent's initial positions and velocities.
	(b) is the trajectories of all agents. 
	Each agent's initial and final positions are indicated by a red circle and a blue star, respectively.
	(c) is the graph of $\norm{v_i}{2}$ for each $i\in\Ical$, and the red and blue dashed lines are markers for $\Cab = 2$ and $t=1$, respectively.
	}
	\label{fig1}
\end{figure}

	Figure \ref{fig2} shows a simulation of the case where $\norm{v_i(t)}{2}\to 0$ as $t\to\infty$.
	For this case, we simply consider four agents (that is $N=4$), and 
	the initial configuration is given by
	\begin{align*}
		\begin{array}{llll}
			x_1 = (1,0), \quad &x_2 = (0,1), \quad &x_3 = (-1,0), \quad &x_4 = (0,-1),\\
	  		v_1 = (-1,0), \quad &v_2 = (0,-1), \quad &v_3 = (1,0), \quad &v_4 = (0,1).
		\end{array}
	\end{align*}
	As shown in Figure \ref{fig2} (c), we can see that $\norm{v}{2}$ converges to zero where $\norm{v}{2}:= \left(\sum_{i\in\Ical} \norm{v_i}{2}^2\right)^{\frac{1}{2}}$. Although it looks like $\norm{v(t)}{2}=0$ around $t=1$ in Figure \ref{fig2} (c), it is not numerically zero. 
	That is, in this case, we cannot determine whether $\norm{v(t)}{2}=0$ in finite time this simulation. 
	
\begin{figure}[thb!]
	\includegraphics[scale=0.35]{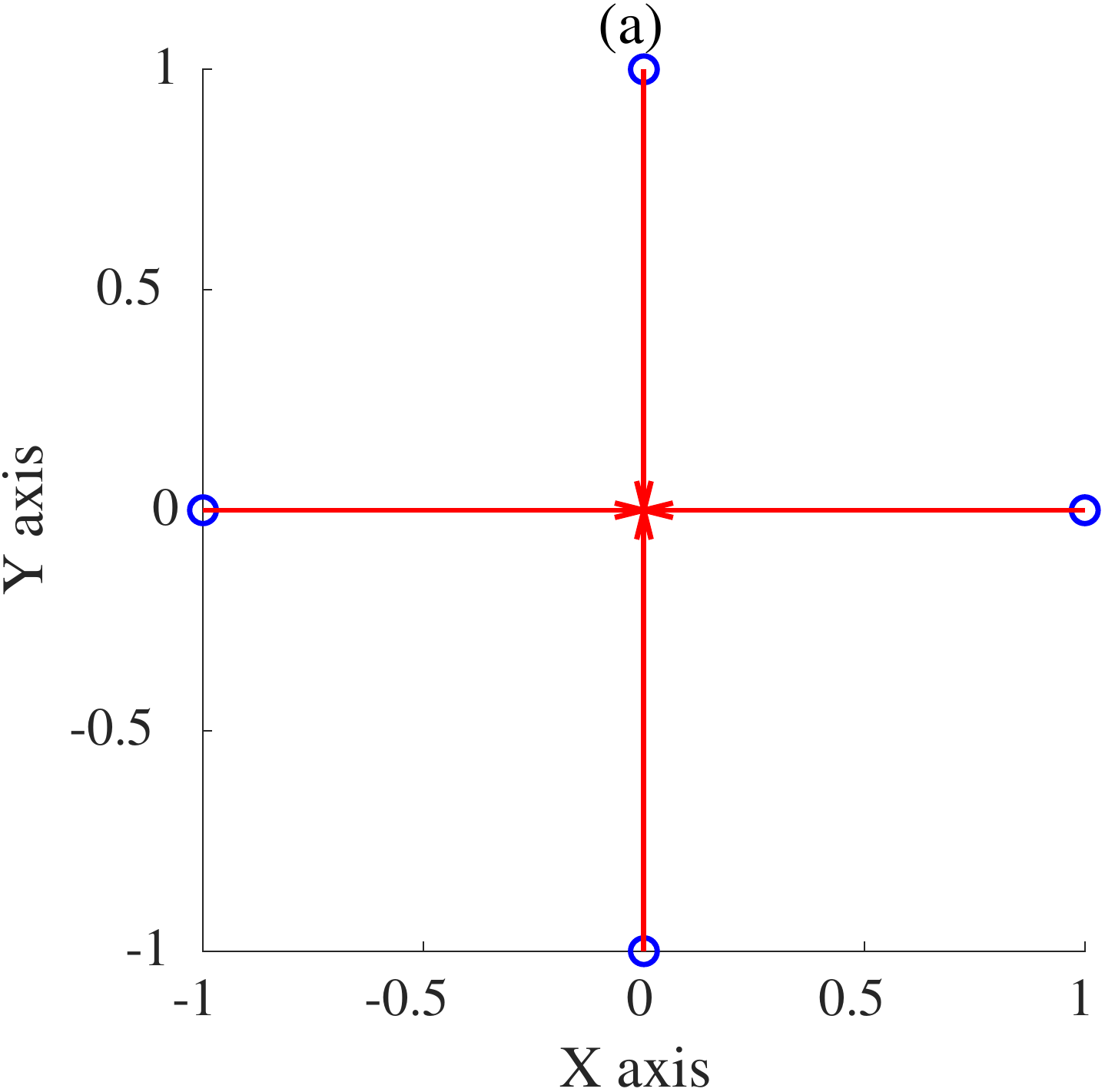}
	\includegraphics[scale=0.35]{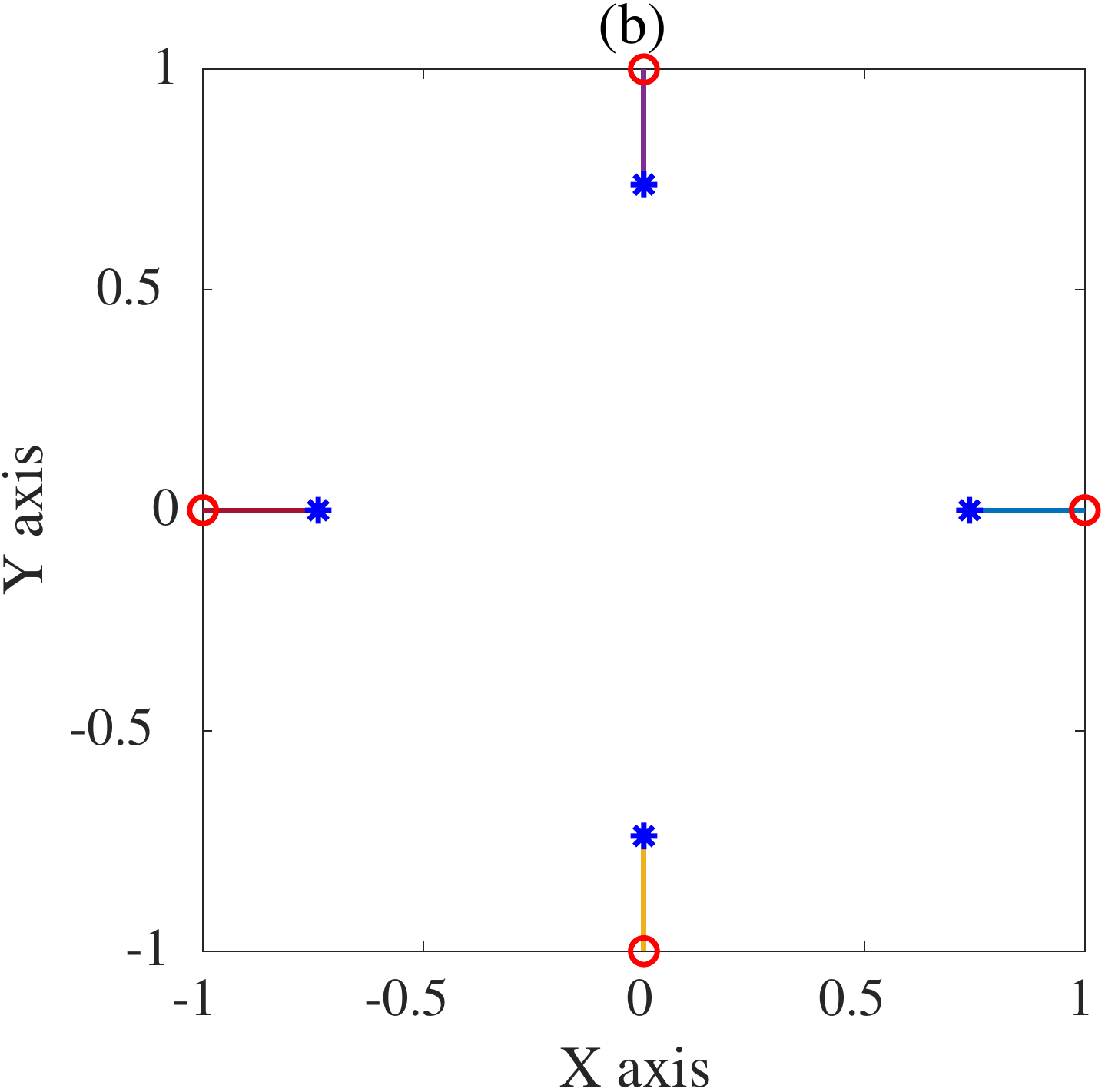}
	\includegraphics[scale=0.35]{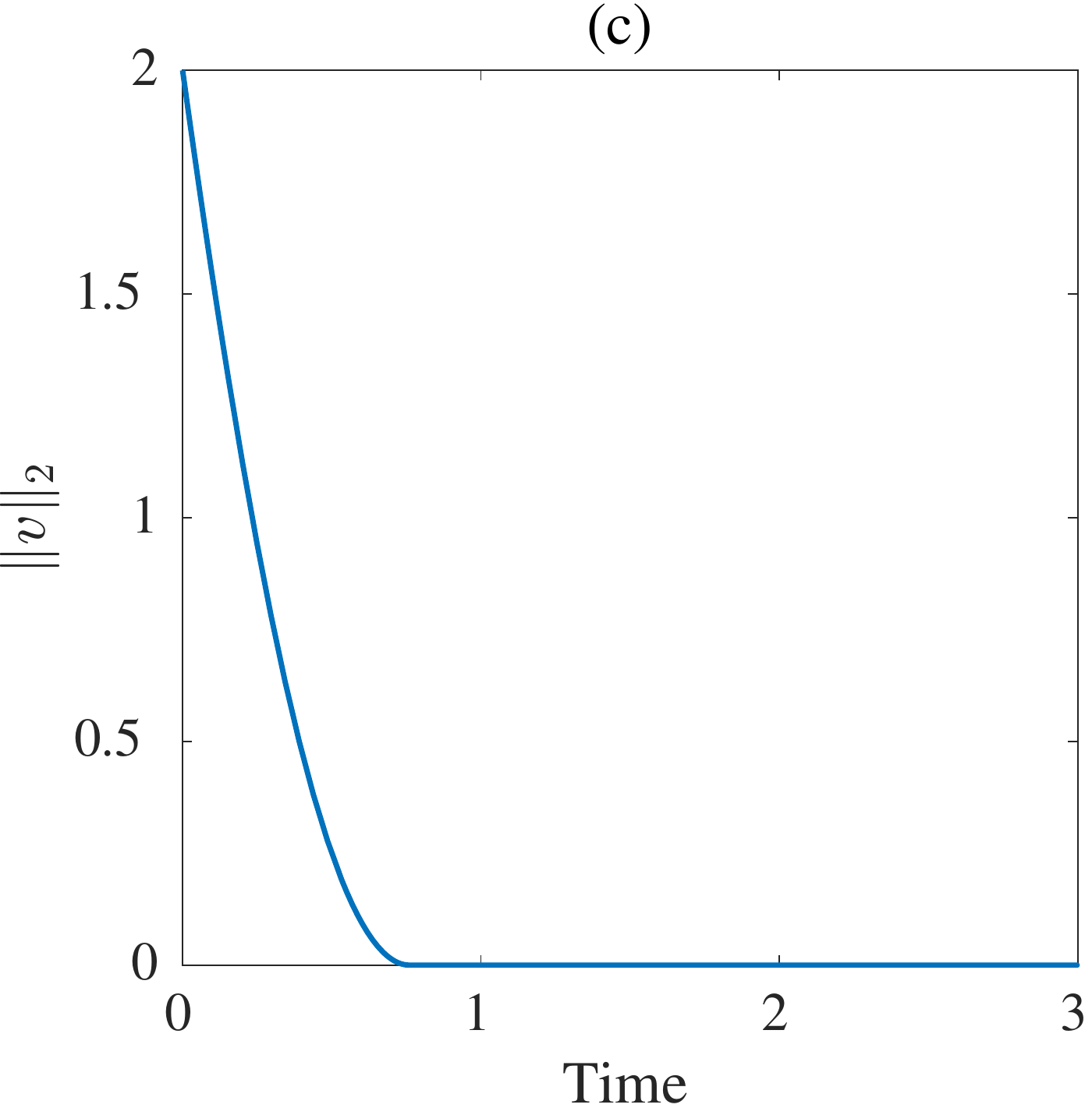}
	\caption{
	(a) The initial configuration.  
	(b) The trajectories of all agents.
	(c) The graph of $\norm{v}{2}$, where $\norm{v}{2}:= \left(\sum_{i\in\Ical} \norm{v_i}{2}^2\right)^{\frac{1}{2}}$.}
	\label{fig2}
\end{figure}	
	
	By contrast, if the velocity of the fourth agent changes from $(0, -1)$ to $(0, -0.1)$, then by Proposition \ref{[Prop] Positive Convergence}, the norm $\norm{v_i}{2}$ must converge to $2$ for all $i\in\Ical$, which can be seen in Figure \ref{fig3}.
	It should be noted that all agents are moving in the negative direction of the $y$-axis, even though the fourth agent's speed along the $y$-axis decreases from 1 to 0.1. 
	
	From this example, we can see that it is not easy to control the direction of movement of the agents with the initial distribution using system \eqref{Rayleigh friction (1)}-\eqref{Rayleigh friction (2)}.

\begin{figure}[thb!]
	\includegraphics[scale=0.35]{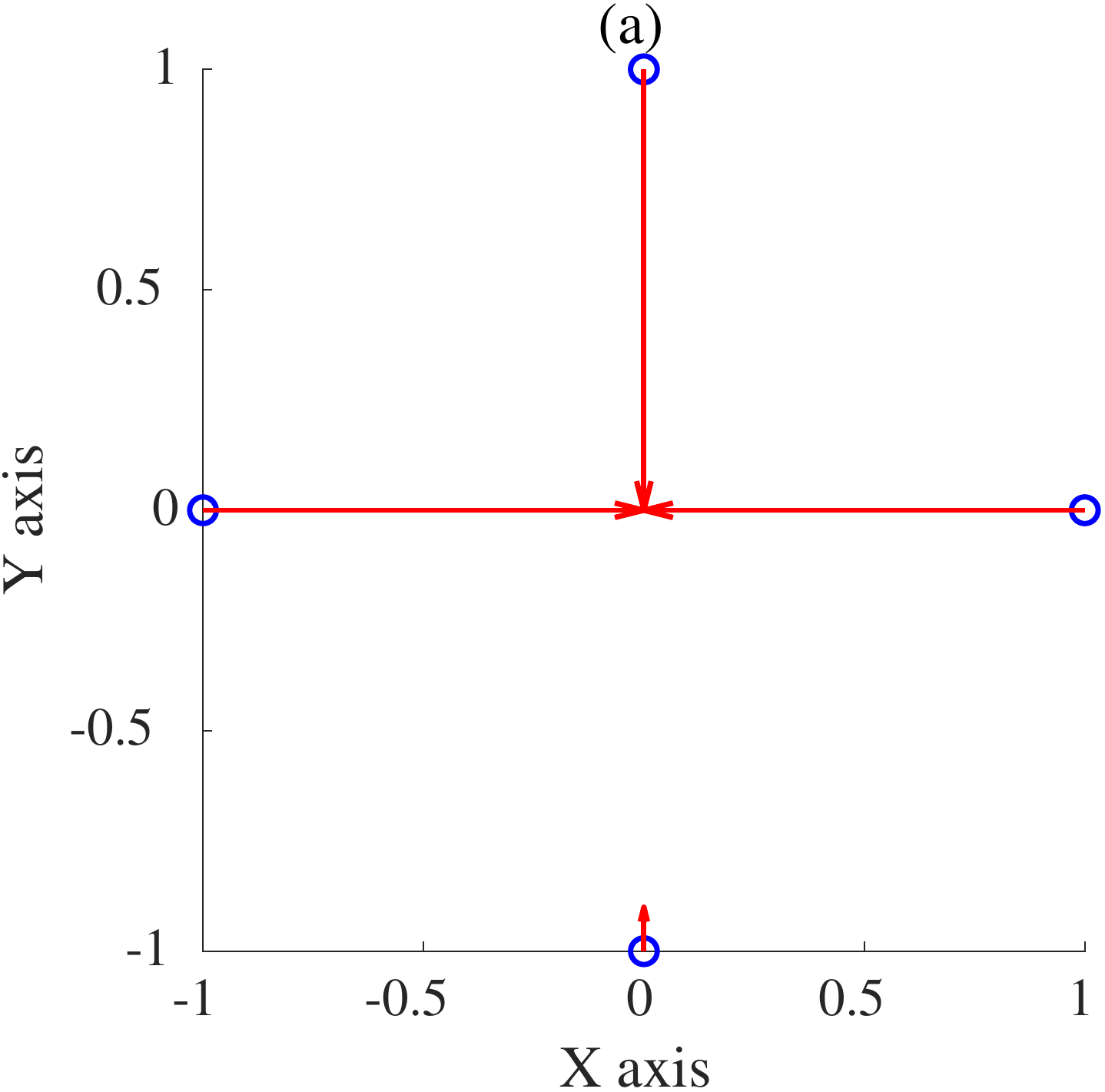}
	\includegraphics[scale=0.35]{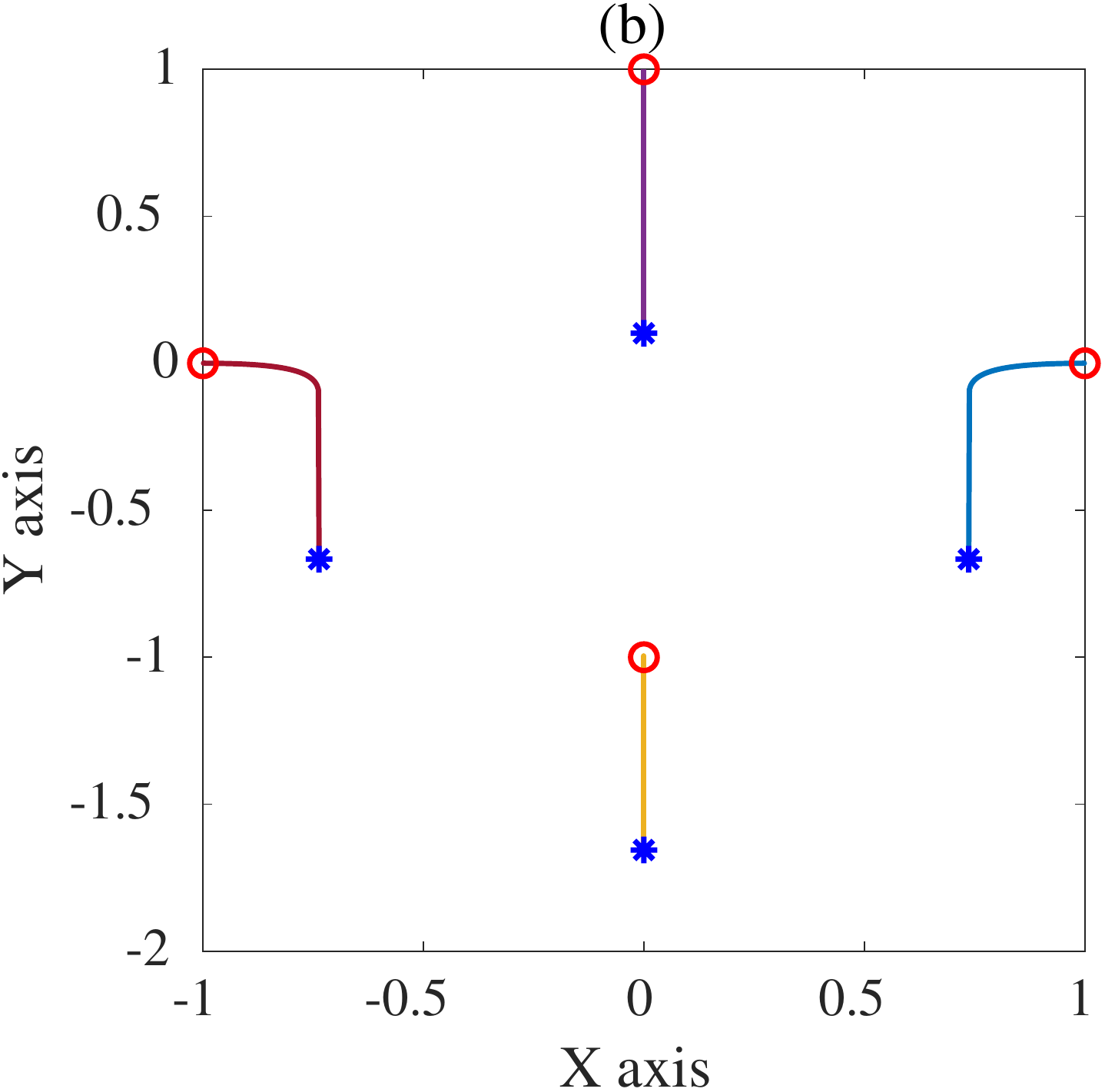}
	\includegraphics[scale=0.35]{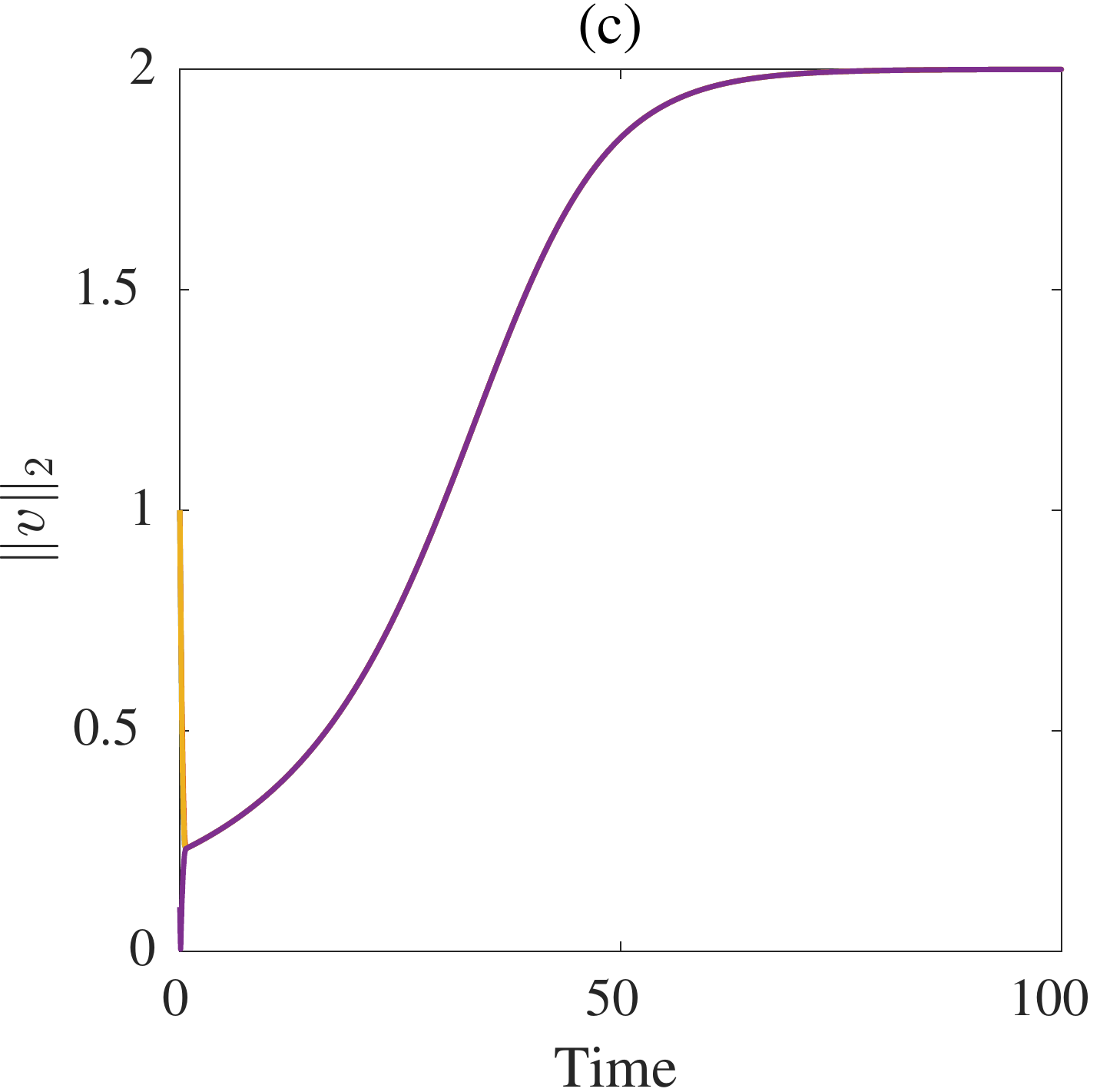}
	\caption{
	(a) The initial configuration.  
	(b) The trajectories of all agents.
	(c) The graph of $\norm{v_i}{2}$.}
	\label{fig3}
\end{figure}

\end{example}

\begin{example}
	In Section 4, we showed that system \MSys{} allows us to control the direction of each agent's velocity from Theorem \ref{Thm Finite Extinction Time} and Proposition \ref{[Prop] Convergence to Cabk}.
	Here, we simulate the results of Theorem \ref{Thm Finite Extinction Time} and Proposition \ref{[Prop] Convergence to Cabk}.
	
	We first consider 20 agents (that is $N=20$), and the same initial distribution as in Example \ref{Ex1}.
	We also provide $a_k$ and $b_k$ for $k=1,2$ as follows:
	\begin{align*}
  		a_1 = 0.1,\quad a_2 = 0.01, \quad b_1 = 0.05, \quad b_2 = 0.1.
	\end{align*}
	Then, we can see that $\lim_{t\to\infty} v_{i,1} (t) = \left( \frac{a_1}{b_1} \right)^{\frac{1}{r-q}}  = 2$ and $\lim_{t\to\infty} v_{i,1} (t) = \left( \frac{a_2}{b_2} \right)^{\frac{1}{r-q}}  = 0.1$ as in Figure \ref{fig4} (c) and (d).
	The finite flocking time is around $t=1$.

	\begin{figure}[thb!]
		\centering
		\includegraphics[scale=0.4]{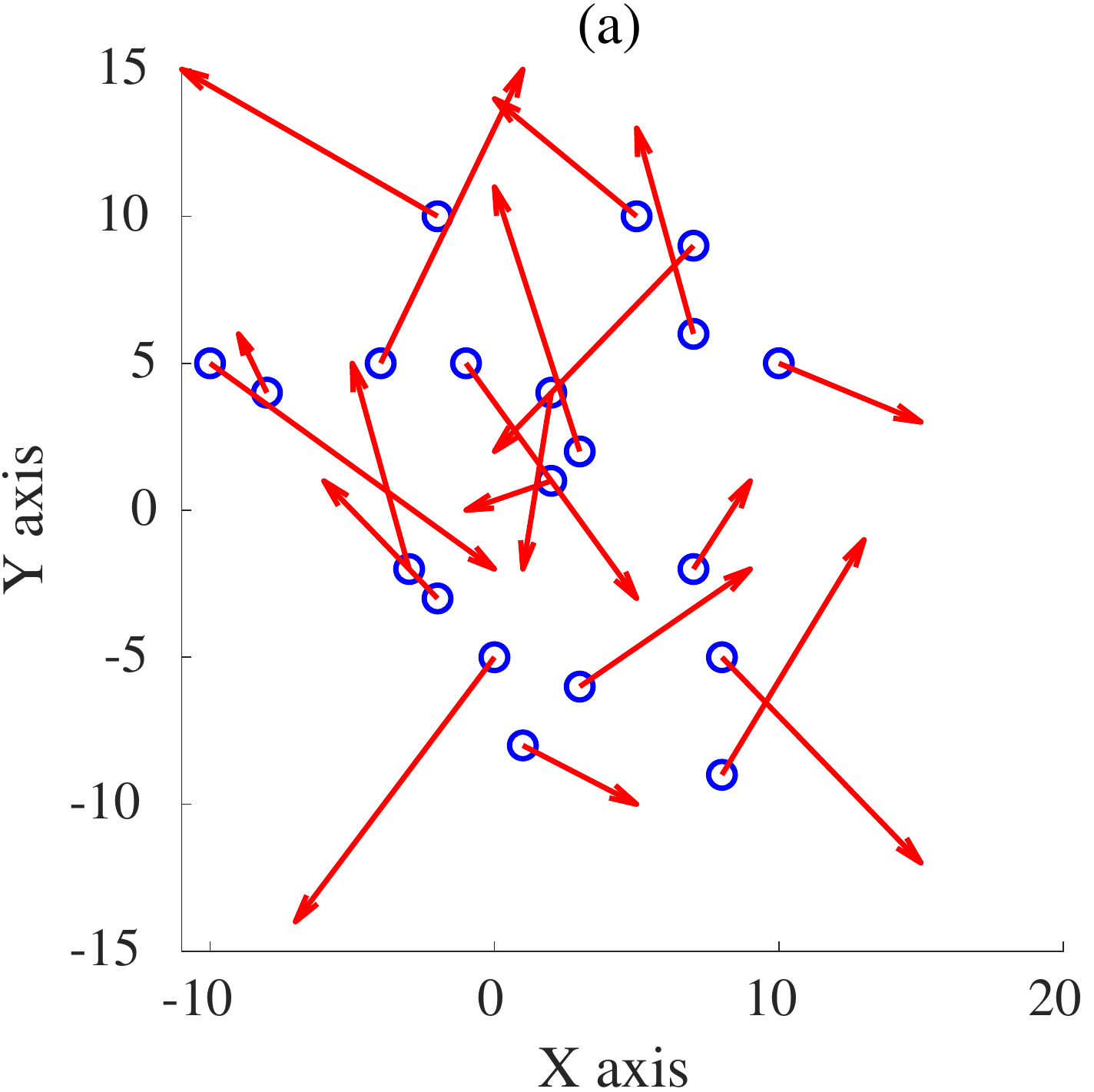}\quad
		\includegraphics[scale=0.4]{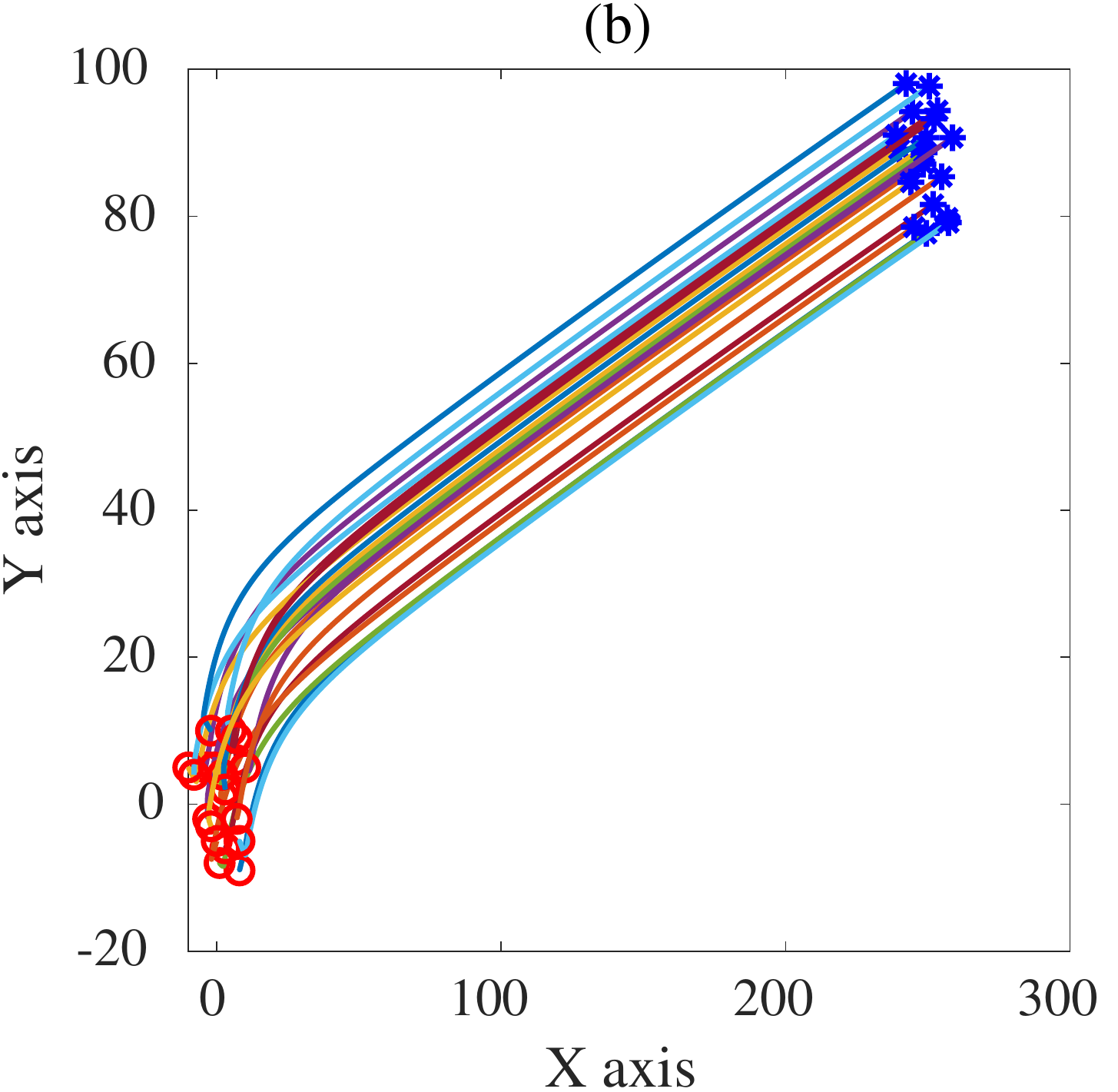}\\
		\includegraphics[scale=0.4]{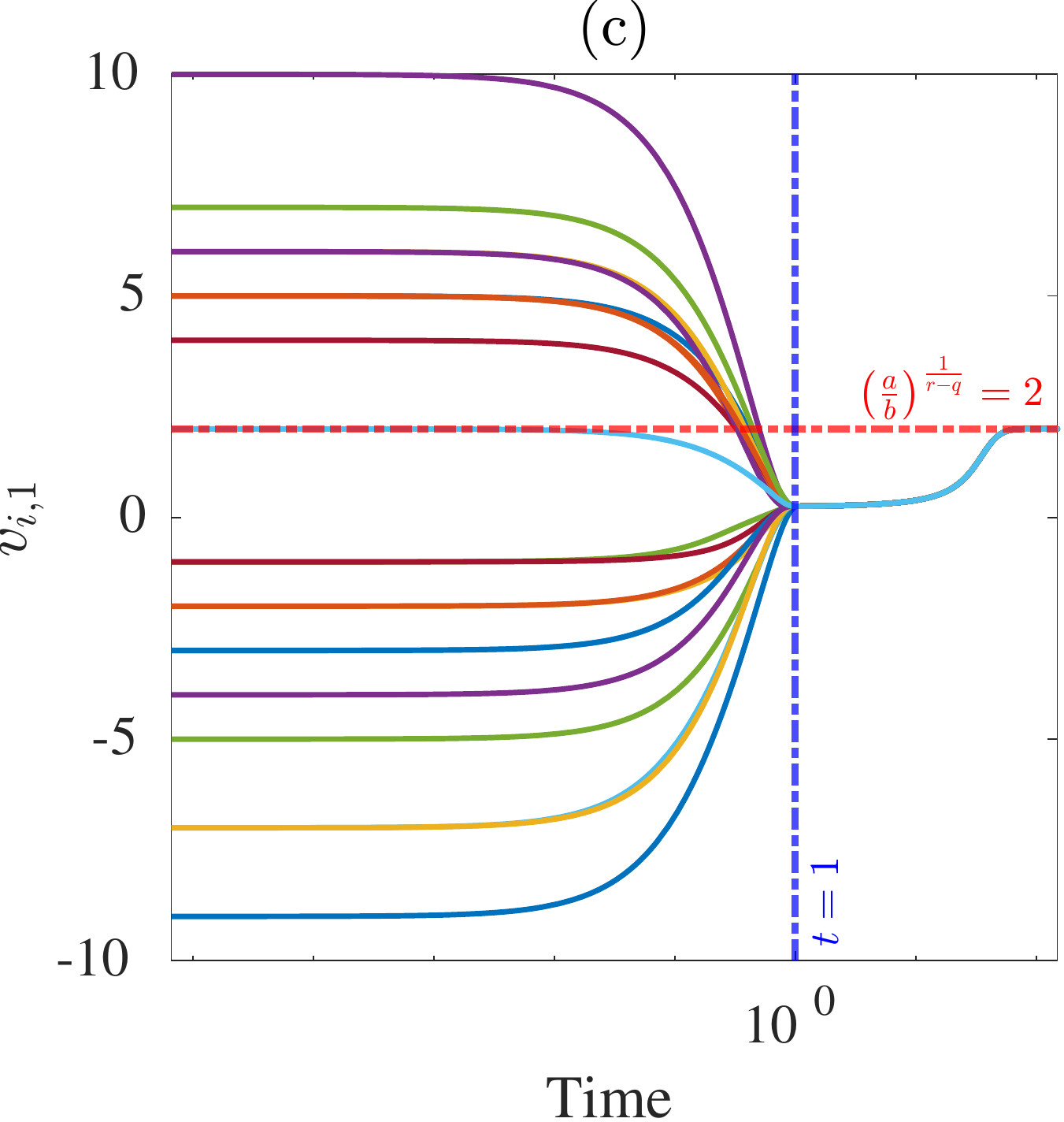}\quad
		\includegraphics[scale=0.4]{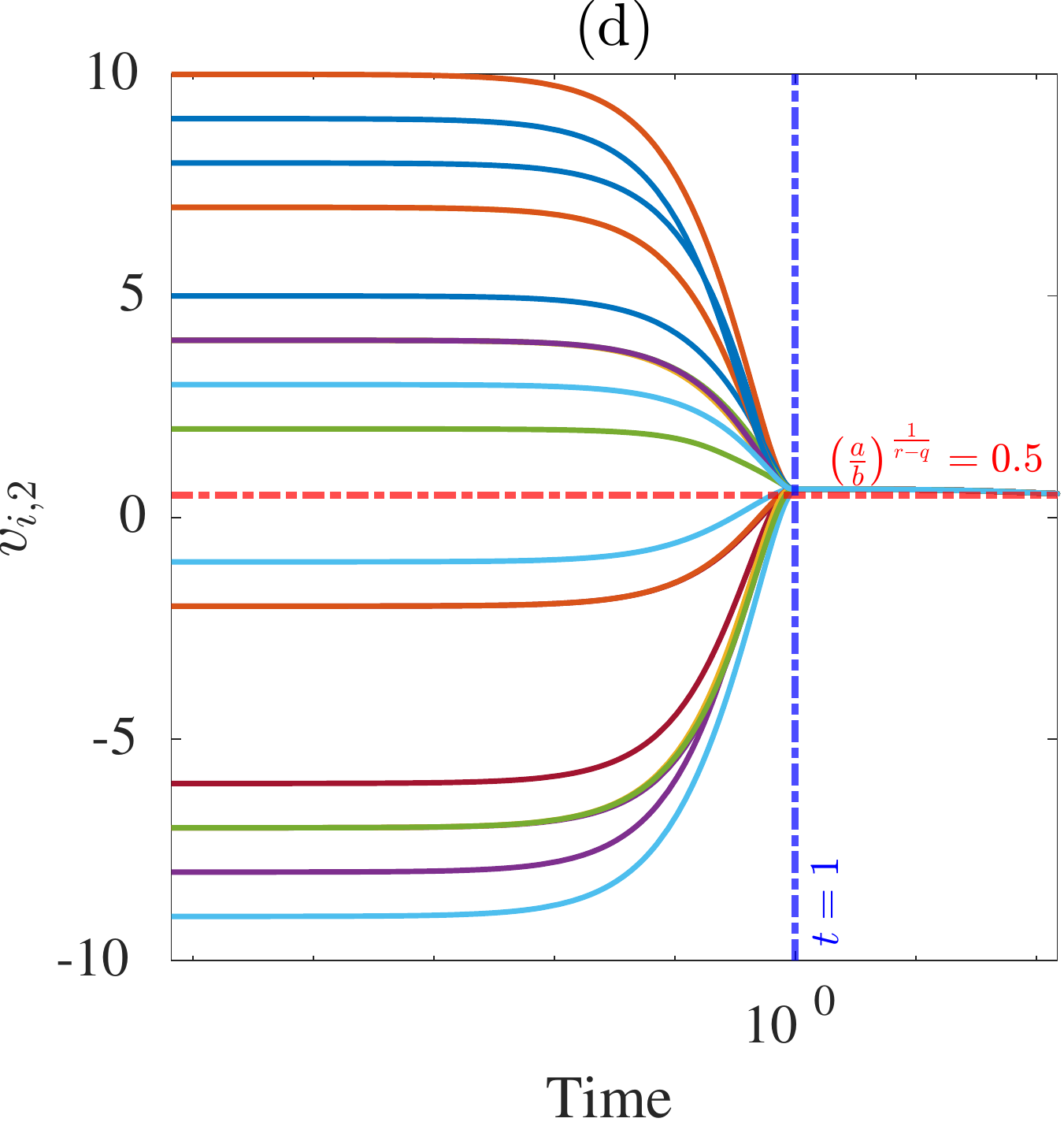}
		\caption{
		(a) The initial configuration.  
		(b) The trajectories of all agents.
		(c) The graph of $v_{i,1}$.
		(d) The graph of $v_{i,2}$.}
		\label{fig4}
	\end{figure}	
	
	Similar to the result of Proposition \ref{[Prop] Convergence to Cabk}, we now control the direction of the agents by adjusting the minimum value $v_{m,2}(0)$ of the initial velocity configuration for the second dimension (that is $k=2$).
	We first take $v_{i,2}(0) = \min \left\{ v_{i,2}, 0.5 \right\}$ for all $i\in\Ical$ as in Figure \ref{fig5} (b). 
	Figure \ref{fig5} (a) is a redrawing of the velocities of the agents indicated by the arrows in Figure \ref{fig4} (a).
	As shown in Figure \ref{fig6}, we can see $\lim_{t\to\infty} v_{i,2} (t) = - \left( \frac{a_1}{b_1} \right)^{\frac{1}{r-q}}  = -0.5$.

	\begin{figure}[thb!]
		\centering
		\includegraphics[scale=0.4]{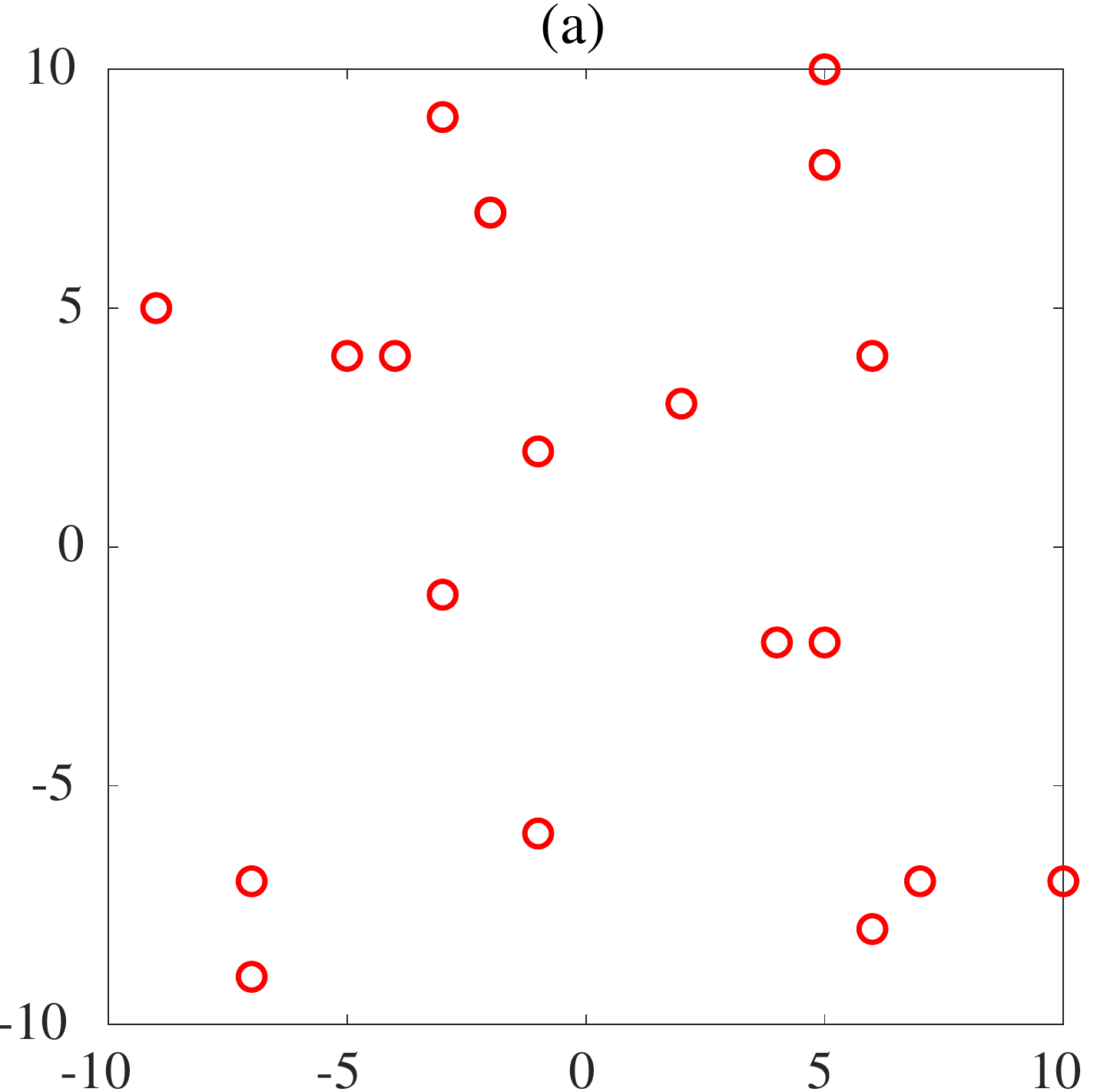}\quad
		\includegraphics[scale=0.4]{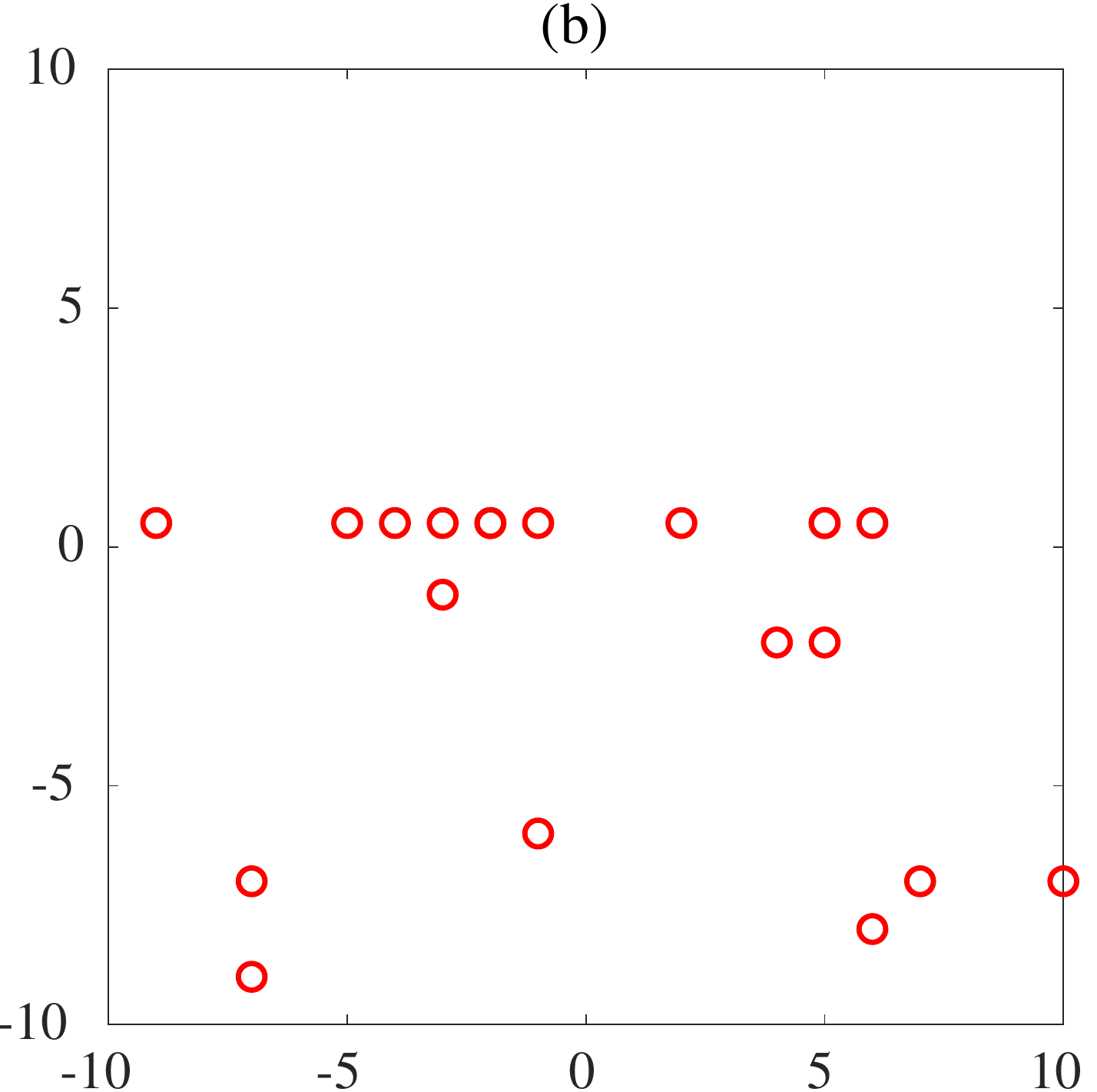}
		\caption{
		(a) The initial velocities of agents in Figure \ref{fig4}.
		(b) The initial velocities to control the direction of the agents}
		\label{fig5}
	\end{figure}

	\begin{figure}[thb!]
		\centering
		\includegraphics[scale=0.4]{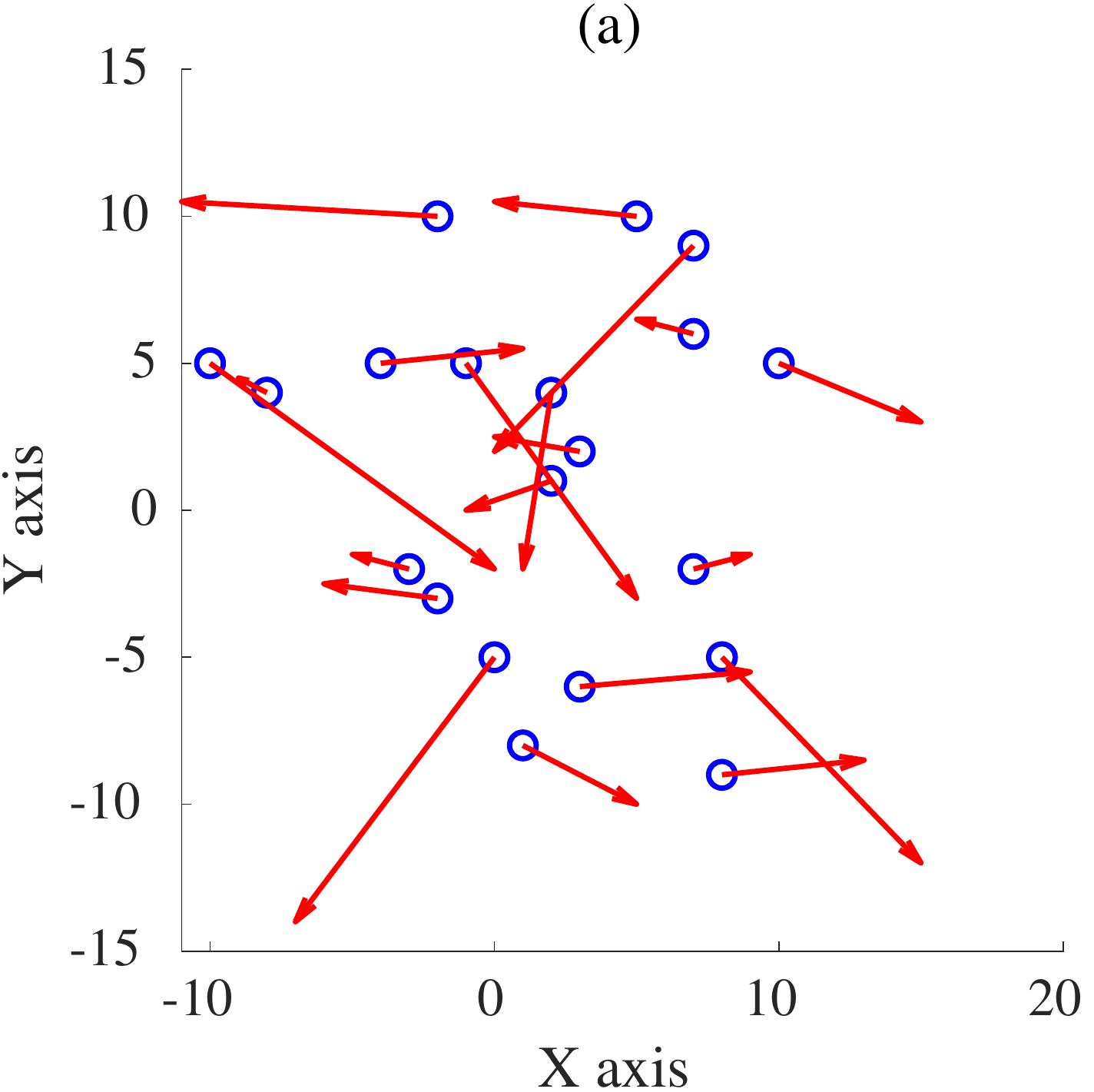}\quad
		\includegraphics[scale=0.4]{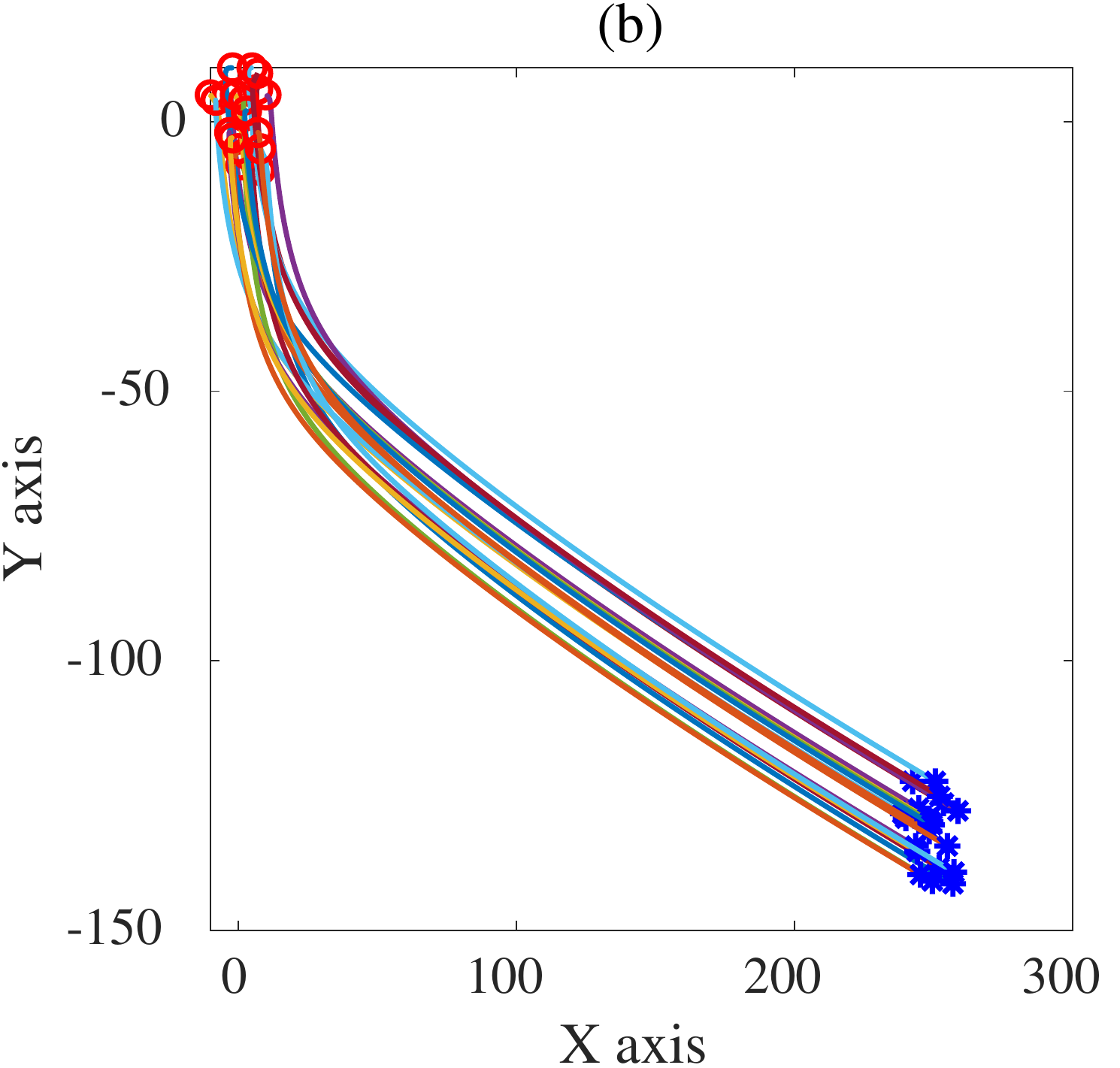}\\
		\includegraphics[scale=0.4]{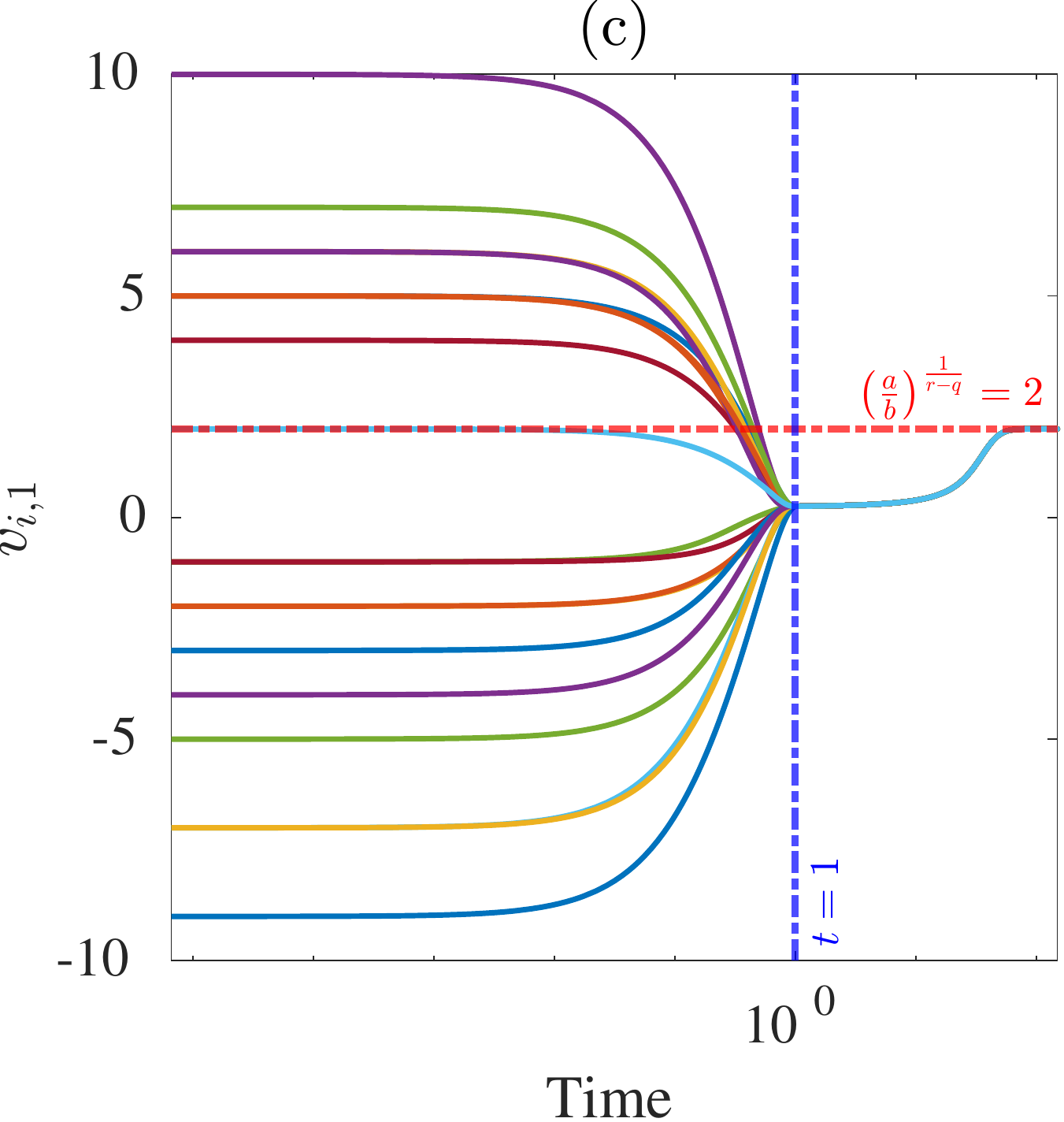}\quad
		\includegraphics[scale=0.4]{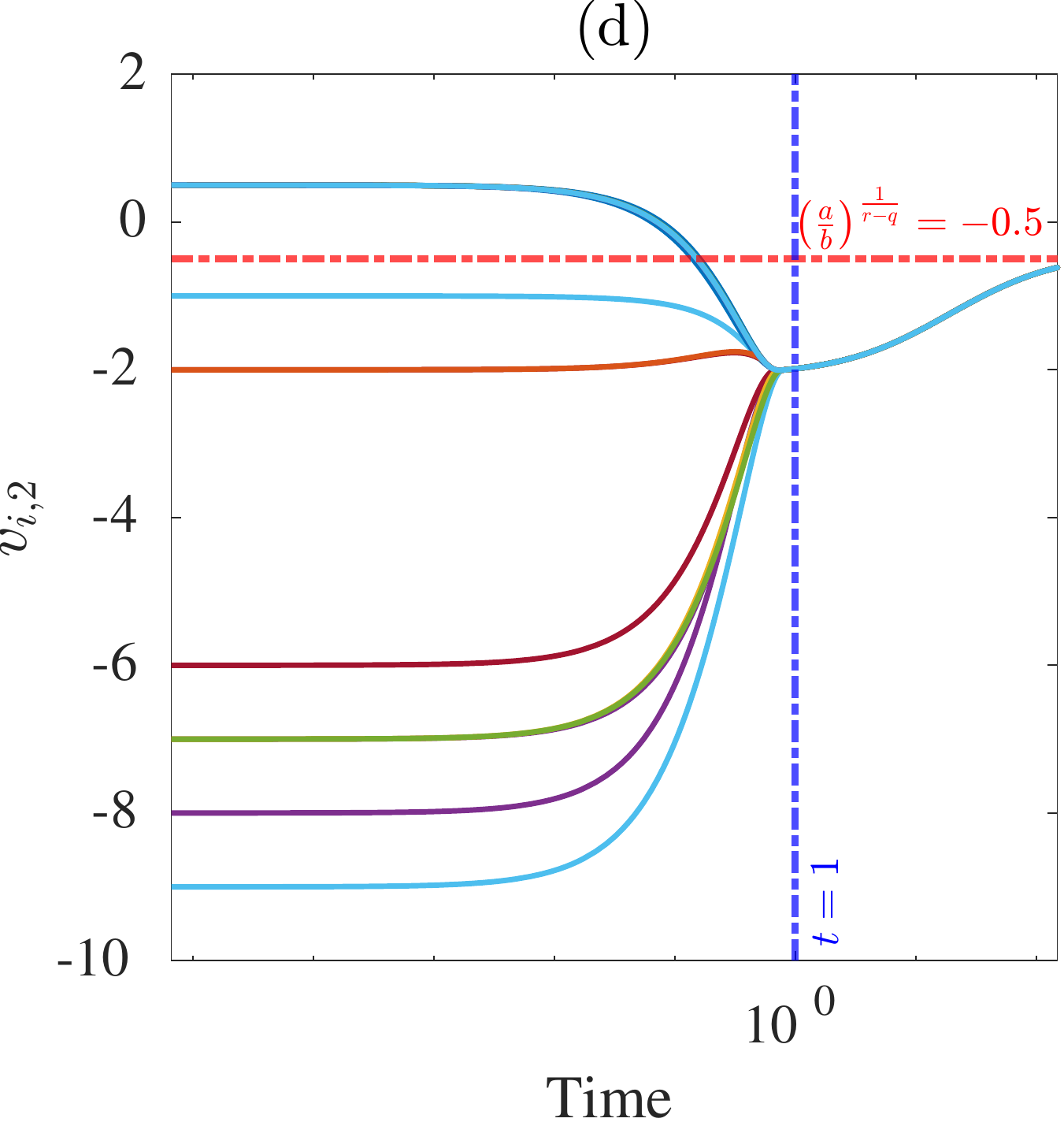}
		\caption{
		(a) The new initial distribution.
		(b) The trajectories of all agents.
		(c) The graph of $v_{i,1}$.
		(d) The graph of $v_{i,2}$.}
		\label{fig6}
	\end{figure}
\end{example}

\section{Conclusion}
In this study, we dicussed three open questions of the Cucker--Smale model with norm-type Rayleigh friction that were addressed in \cite{ha2010asymptotic}.
For these open questions, we proposed a more generalized model, the nonlinear Cucker--Smale model with norm-type Rayleigh friction, using the discrete $p$-Laplacian ($1<p<2$).
For this model, we presented some conditions to guarantee that the norm $\norm{v_i}{2}$ converges to 0 or $\Cab$ as $t\to\infty$. 
Moreover, we also presented conditions under which the norm $\norm{v_i}{2}$ converges to only $\Cab$ as $t\to\infty$. 
We also proposed the nonlinear Cucker--Smale model with vector-type Rayleigh friction to discuss the open question about the direction of the velocities of agents and presented conditions to control the direction of the agents’ velocities by parameters in the model.
Finally, we also showed that the regular communication weight $\psi_R$ satisfies the conditions given in this paper.
In particular, we presented in this paper conditions that the parameter $a$ and $|v_{m,k}(0)|$ must be very small. However, we leave the question of how small they should be as a topic for future research.

\section*{Acknowledgements}
The second author was supported in part by NSF-DMS 2208499 and REP grant for the year of 2022 from Texas State University.
The corresponding author was supported by a National Research Foundation of Korea(NRF) grant funded by the Korean government(MSIT) (No. 2021R1A2C1093929), and Kunsan National University in 2022.

\bibliographystyle{elsarticle-num}
\bibliography{NCS_RF.bib}

\end{document}